\documentclass[12pt,leqno]{amsart}
\usepackage{amssymb,amsmath}

\numberwithin{equation}{section}

\newcommand{\beq}{\begin{equation}}
\newcommand{\eeq}{\end{equation}}

\newcommand{\bpf}{\begin{proof}}
\newcommand{\epf}{\end{proof}}
\newcommand{\bsp}{\begin{split}}
\newcommand{\esp}{\end{split}}
\newcommand{\bg}{\begin{gathered}}
\newcommand{\eg}{\end{gathered}}
\newtheorem{Thm}{Theorem}[section]
\newtheorem{Lem}{Lemma}[section]
\newtheorem{Cor}{Corollary}[section]
\newtheorem{Rem}{Remark}[section]
\newtheorem{Prop}{Proposition}[section]
\newtheorem{Def}{Definition}[section]

\newcommand{\x}{\xi}
\newcommand{\g}{\eta}
\newcommand{\h}{\zeta}

\newcommand{\absbalv}{\rho_1}

\newcommand{\Real}{\mathbb{R}}
\newcommand{\B}{\mathcal{B}}
\newcommand{\dt}{k}
\newcommand{\al}{\alpha}
\newcommand{\te}{\theta}
\def\T {{\mathbb T}}
\def\N {\mathbb{N}}
\def\C {{\mathcal C}}
\def\dist{{\rm dist}}
\def\A {{\mathcal A}}
\def\K {{\mathcal K}}
\begin{document}

\title
[Long-term Dynamics of 2D
Thermohydraulics Equations]
{Approximation of the Long-term Dynamics of
the Dynamical System Generated by the Two-dimensional
Thermohydraulics Equations}

\author[Tone]{Florentina ~Tone}

\date{August 24, 2011}

\begin{abstract}
Pursuing our work in \cite{T4}, \cite{T3}, \cite{TW}, \cite{CZT},  we consider in this article the two-dimensional thermohydraulics equations. We discretize these equations in time using the implicit
Euler scheme and  
we prove that the 
global attractors generated by the numerical scheme converge to the global attractor of the continuous system as the time-step approaches zero.

\end{abstract}



\maketitle

\tableofcontents


\section{Introduction}\label{s:intro}
In this article we discretize  the two-dimensional thermohydraulics equations in time  using the implicit Euler scheme, and we show that 
global attractors generated by the numerical scheme converge to the global attractor of the continuous system as the time-step approaches zero.
In order to do this, we first prove that the scheme is $H^1$-uniformly stable in time (see Section \ref{s4}) and then we show that the long-term dynamics of  the continuous system can be approximated by the discrete attractors of the dynamical systems generated by the numerical scheme (see Section \ref{s5}).

In the case of the Navier--Stokes equations with Dirichlet boundary conditions, 
the $H^1$-uniform stability of the fully implicit Euler scheme has proven to be rather challenging. However, using techniques based on the classical and uniform discrete Gronwall lemmas, we have been able to show the $H^1$-stability for all time of the implicit Euler scheme for the Navier--Stokes equations with Dirichlet boundary conditions (see \cite{TW}). The $H^2$-stability has also been established. More precisely, the $H^2$-stability has first been proven in the simpler case of space periodic boundary conditions (see \cite{T3}), and then extended to Dirichlet boundary conditions (see \cite{T4}); the magnetohydrodynamics equations are also considered in \cite{T4}.

Our first objective in this article is to extend the $H^1$-uniform stability proven in \cite{TW} for the Navier--Stokes equations with Dirichlet boundary conditions, to the thermohydraulics equations. In order to do so, we divide our proof into three steps. First, we prove the $L^2$-uniform stability of both the discrete velocity $v^n$ and the discrete temperature $\te^n$ (see Lemma \ref{max} below). Then, using techniques based on the classical and uniform discrete Gronwall lemmas, we derive the $H^1$-uniform stability of $v^n$ (see Theorem \ref{mainth} below), which we will use in Subsection \ref{ss:vte} in order to establish the $H^1$-uniform stability of $\te^n$ (see Theorem \ref{t:vte} below). Besides the intrinsec interest of considering the thermohydraulics equations, the new technical difficulties which appear here are related to the specific treatment of the temperature with the necessary utilization of the maximum principle. Furthermore, we have simplified some steps of the proof
as compared to \cite{TW}.

Our second objective in this article is to employ the technique developed in \cite{CZT} to prove that the global attractors generated by the fully implicit Euler scheme converge to the global attractor
of the continuous system as the time-step approaches zero. When discretizing the two-dimensional thermohydraulics equations in time  using the implicit Euler scheme,
one can prove the uniqueness of the solution provided that 
the time step is sufficiently small. More precisely, the time
restriction depends on the initial value, and thus one cannot define a single-valued attractor in the classical sense. This is why we need to use the theory of the so-called multi-valued attractors,
which we briefly recall in Subsection \ref{ssec:abs}.


\section{The thermohydraulics equations}\label{s:Thermo}
Let $\Omega=(0,1) \times (0,1)$ be the domain occupied by the fluid and let $e_2$ be the unit upward vertical vector. The thermohydraulics equations consist of the coupled system of the equations of fluid and temperature in the Boussinesq approximation and they read (see, e.g., \cite{FMT}, \cite{temam:iddsmp}):
\begin{gather}
   \frac{ \partial v}{\partial t} + (v \cdot \nabla) v - \nu \Delta v + \nabla p = e_2(T-T_1),\label{1.1}\\
   \frac{ \partial T}{\partial t} + (v \cdot \nabla) T - \kappa  \Delta T  = 0, \label{1.2}\\
   \text{div}\, v = 0;\label{1.3}
\end{gather}
here $v=(v_1,v_2)$ is the velocity, $p$ is the pressure, $T$ is the
temperature, $T_1$ is the temperature at the top boundary, $x_2=1$, and $\nu$, $\kappa$ are positive constants. We supplement these equations with the initial conditions
\begin{gather}
   v(x,0) = v_0(x), \label{1.4}\\
   T(x,0)=T^0(x), \label{1.5}
\end{gather}
where $v_0:\Omega \to \mathbb{R}^2$, $T^0:\Omega \to \mathbb{R}$
are given, and with the boundary conditions
\begin{gather}
   v=0 \quad\textrm{at}\quad x_2=0 \quad\text{ and } \quad x_2=1, \label{1.6}\\
   T =T_0=T_1+1 \quad\textrm{at} \quad x_2=0 \quad \text{ and } \quad T=T_1 \quad \textrm{at} \quad x_2=1,  \label{1.7}
\end{gather}
and
\begin{equation}\label{1.8}
\begin{split}
p, v, T & \text { and the first derivatives of } v \text{ and } T \text{ are periodic} \\ &\text{ of period } 1 \text{ in the direction } x_1,
\end{split}
\end{equation}
meaning that $\phi \arrowvert_{x_1=0}=\phi \arrowvert_ {x_1=1}$ for the corresponding functions $\phi$.

Letting
\begin{equation}\label{1.9}
\theta=T-T_0+x_2,
\end{equation}
and changing $p$ to
\begin{equation}\label{1.10}
p-\left( x_2-\frac{x_2^2}{2} \right),
\end{equation}
equations \eqref{1.1}--\eqref{1.3} together with the boundary conditions \eqref{1.6}--\eqref{1.8} become
\begin{gather}
   \frac{ \partial v}{\partial t} + (v \cdot \nabla) v - \nu \Delta v + \nabla p = e_2 \theta,\label{1.11}\\
   \frac{ \partial \theta}{\partial t} + (v \cdot \nabla)  \theta - v_2- \kappa  \Delta  \theta  = 0, \label{1.12}\\
   \text{div}\, v = 0,\label{1.13}\\
   v=0 \quad\textrm{at} \quad x_2=0 \quad \text{ and }  \quad x_2=1,  \label{1.14a}\\
   \theta=0 \quad\textrm{at} \quad x_2=0 \quad \text{ and }  \quad x_2=1,  \label{1.14}\\
    \eqref{1.8}  \text{ holds with } T  \text{ replaced by } \theta. \label{1.14a}
\end{gather}
These equations are supplemented with the initial conditions
\begin{gather}
   v(x,0) = v_0(x), \label{1.14b}\\
   \te(x,0)=T^0(x)-T_0+x_2=:\te_0(x). \label{1.14c}
\end{gather}
For the mathematical setting of the problem we define the space $H=H_1 \times H_2$, where
\begin{gather}
H_1 = \left\{ v \in L^{2}(\Omega)^{2},\, \text{div}=0, \,  v_2 \arrowvert_{x_2=0}=v_2 \arrowvert_{x_2=1}=0, v_1 \arrowvert_{x_1=0}=v_1 \arrowvert_{x_1=1}
       \right\},\label{1.15}\\
H_2=L^{2}(\Omega),\label{1.16}
\end{gather}
and we denote the scalar products and norms in $H_1$,  $H_2$ and $H$ by $(\cdot, \cdot)$ and $| \cdot |$.

We also define the space $V=V_1 \times V_2$, where
\begin{gather}
V_1=\left\{ v \in H^{1}(\Omega)^2,\,  v \arrowvert_{x_2=0}=v \arrowvert_{x_2=1}=0, v \arrowvert_{x_1=0}=v \arrowvert_{x_1=1}, \text{ div } v=0 \right\},\label{1.18}\\
V_2 = \left\{ \theta \in H^{1}(\Omega),\,  \theta \arrowvert_{x_2=0}=\theta \arrowvert_{x_2=1}=0, \theta \arrowvert_{x_1=0}=\theta \arrowvert_{x_1=1} \right\}.\label{1.17}
\end{gather}
The space $V_2$ is a Hilbert space with the scalar product and the norm
\begin{equation}\label{1.19}
((\phi, \psi))=\int \nabla \phi \cdot \nabla \psi \, dx, \qquad \| \phi \|=\sqrt{((\phi,\phi))},
\end{equation}
and we have the Poincar\'e inequality (see, e.g., \cite{temam:iddsmp}, page $134$)
\beq\label{Poin}
   |\phi| \leq  \|\phi\|, \quad \forall\, \phi \in V_1 \text{ or } V_2.
\eeq
We denote both scalar products and norms in $V_1$ and $V$ by $(( \cdot, \cdot))$ and $\| \cdot \|$.

Let $D(A)=D(A_1) \times D(A_2)$, where
\begin{equation}
D(A_i)=\left \{ v \in V_i \cap H^2(\Omega)^2 , \frac{ \partial v}{ \partial x_1} \Big | _{x_1=0} =\frac{ \partial v}{ \partial x_1} \Big|_{x_1=1}\right\}, \, i=1,2,
\end{equation}
and let $A$ be the linear operator from $D(A)$ into $H$ and from $V$ into $V'$ defined by
\begin{equation}\label{1.20}
(Au_1, u_2)=a(u_1, u_2), \, \forall \, u_i=\{v_i,\theta_i\} \in D(A), \,i=1,2,
\end{equation}
with
\begin{equation}\label{1.21}
a(u_1,u_2)=\nu((v_1,v_2))+\kappa((\theta_1, \theta_2)).
\end{equation}

We consider the trilinear continuous form $b$ on $V$, defined by
\begin{equation}\label{1.22}
\begin{split}
 b(u_1, u_2, u_3)= &b_1(v_1, v_2, v_3) + b_2(v_1, \theta_2, \theta_3), \forall u_i=\{v_i,\theta_i\} \in V,
\end{split}
\end{equation}
where
\begin{equation}\label{1.23}
   b_1(y,w,z) = \sum_{i,j=1,2}\> \int_\Omega y_i
        \frac{\partial w_j}{\partial x_i} z_j \;{\rm d}x, \forall \, y, w, z \in H^1(\Omega)^2,
\end{equation}
\begin{equation}\label{1.24}
   b_2(y,\phi,\psi) = \sum_{i=1}^2 \> \int_\Omega y_i
        \frac{\partial \phi}{\partial x_i} \psi \;{\rm d}x, \forall \, y \in H^1(\Omega)^2,  \phi,\psi \in H^1(\Omega).
\end{equation}
The form $b_1$ is trilinear continuous on $V_1 \times V_1 \times V_1$
and enjoys the following properties:
\beq\label{b1.1}
\begin{gathered}
   |b_1(y,w,z)| \le c_b |y|^{1/2} \|y\|^{1/2} \|w\| |z|^{1/2} \|z\|^{1/2},
   \quad\forall \, y, w, z  \in V_1,
\end{gathered}
\eeq \beq\label{b2.1}
\begin{gathered}
   |b_1(y, w, z)| \leq c_b |y|^{1/2} | A_1 y|^{1/2} \| w \| |z|,\\
   \qquad\forall \,y \in D(A_1), \, w\in V_1, \, z  \in H_1,
\end{gathered}
\eeq
\begin{equation}\label{b3.1}
\begin{gathered}
   |b_1(y,w,z)| \le c_b |y|^{1/2} \|y\|^{1/2} \|w\|^{1/2} |A_1 w|^{1/2} |z|,\\
   \qquad \forall \, y \in V_1, w \in D(A_1), z \in H_1,
\end{gathered}
\eeq \beq\label{b4.1}
   b_1(y,w,w)=0, \quad \forall \, y,w \in V_1,
\eeq the last equation implying
\begin{equation}\label{b5.1}
   b_1(y,w,z)=-b_1(y,z,w), \quad  \forall \, y,w,z \in V_1.
\end{equation}

The form $b_2$ is trilinear continuous on $V_1 \times V_2 \times V_2$
and enjoys the following properties, similar to \eqref{b1.1}--\eqref{b5.1}:
\beq\label{b2}
\begin{gathered}
   |b_2(y,\phi,\psi)| \le c_b |y|^{1/2} \|y\|^{1/2} \|\phi\| |\psi|^{1/2} \|\psi\|^{1/2},
   \quad\forall \, y, \phi, \psi  \in V_2,
\end{gathered}
\eeq \beq\label{b2.2}
\begin{gathered}
   |b_2(y, \phi, \psi)| \leq c_b |y|^{1/2} | A_2 y|^{1/2} \| \phi \| |\psi|,\\
   \qquad\forall \,y \in D(A_2), \, \phi\in V_2, \, \psi  \in H_2,
\end{gathered}
\eeq
\begin{equation}\label{b3.2}
\begin{gathered}
   |b_2(y,\phi,\psi)| \le c_b |y|^{1/2} \|y\|^{1/2} \|\phi\|^{1/2} |A_2 \phi|^{1/2} |\psi|,\\
   \qquad\forall \, y \in V_1, \, \phi \in D(A_2), \, \psi \in H_2,
\end{gathered}
\eeq \beq\label{b4.2}
   b_2(y,\phi,\phi)=0, \quad \forall \, y \in V_1, \, \phi \in V_2,
\eeq the last equation implying
\begin{equation}\label{b5.2}
   b_2(y,\phi,\psi)=-b_2(y,\psi,\phi), \quad \forall \, y \in V_1, \, \phi,\psi \in V_2.
\end{equation}

We associate with $b$ the bilinear continuous operator $B$ from $V
\times V$ into $V'$ and from $D(A) \times D(A)$ into $H$, such that
\begin{equation}\label{B}
   \langle B(u_1,u_2), u_3 \rangle_{V',V}=b(u_1,u_2,u_3),
       \quad \forall\, u_1,u_2,u_3 \in V.
\end{equation}

We also define the continuous operator in $H$
\begin{equation}\label{R}
   Ru=-\{e_2 \te, v_2\}, \, u=\{v,\te\}.
\end{equation}
For more details about the function spaces $D(A)$, $V$ and $H$, as
well as the operators $A$, $B$, $R$ and $b$, the reader is referred to,
e.g., \cite{temam:iddsmp}.

 In the above notation, the system
(\ref{1.11})--(\ref{1.13}) can be written as the functional evolution
equation
\beq\label{q-evolu}
   u_t + Au + B(u) + Ru = 0,
    \qquad u(0) = u_0 = \{ v_0, \te_0 \}.
\eeq

In the two-dimensional case under consideration, the solution to the
thermohydraulics equations is known to be smooth for all time (cf.\ \cite{temam:iddsmp}). Using the maximum principle for parabolic equations,
one can show that $\te \in L^\infty(\Real_+;L^{2}(\Omega))$ and
the velocity $u$ is bounded uniformly for all time by
\beq\label{q-bd-ucts}
   |v(t)|_{L^2(\Omega)^2}^2
      \le {\rm e}^{-\nu  t} |v_0|_{L^2(\Omega)^2}^2
      + \frac{\te_{\infty}^2}{\nu^2} \bigl( 1 - {\rm e}^{- \nu t} \bigr),
\eeq
where $\te_{\infty}=|\te|_{L^{\infty}(\Real_+;L^{2}(\Omega))}$.
Furthermore, using techniques based on the uniform Gronwall lemma
(cf.\ \cite{temam:iddsmp}), one can bound the solution $u$ of \eqref{q-evolu} uniformly in $V$
for all $t\ge0$.

In this article we discretize (\ref{q-evolu}) in time using the
fully implicit Euler scheme, and define recursively the elements $u^n=\{v^n, \te^n\}$ of $V$ as follows:
\beq\label{2.45a}
\begin{split}
u^0=\{v^0, \te^0\}, \text{ where } v^0(x)=v_0(x), \text{ and }\\
\te^0(x)=\te_0(x):=T^0(x)-T_0+x_2 \text{  are given};
\end{split}
\eeq
then when $u^0=\{v^0, \te^0\}, \cdots, u^{n-1}=\{v^{n-1}, \te^{n-1}\}$ are known, we define $u^n=\{v^n, \te^n\} \in V$ such that
\begin{gather}
\frac{1}{\dt} (v^n - v^{n-1}, v)+ \nu ((v^n, v))+b_1(v^n, v^n, v) = (e_2 \theta^n, v), \, \forall v\in V_1,\label{1.25}\\
   \frac{1}{\dt}(\theta^n-\te^{n-1}, \te) + \kappa ((\theta^n, \te)) + b_2(v^n ,\theta^n, \te) - (v_2^n, \te) = 0, \, \forall \te \in V_2.\label{1.26}
\end{gather}

The above system is very similar to the stationary Navier--Stokes
equations and the existence of solutions is proven e.g. by the
Galerkin method, as in \cite{temam:nse}. Uniqueness can also be
derived as in \cite{temam:nse} under some conditions. Let us explain
this point, which somehow motivates the developments in Section
\ref{s5}. For that, we rewrite the system
\eqref{2.45a}--\eqref{1.26} in the form
\begin{gather}
(v^n, v)+ \nu \dt((v^n, v))+\dt b_1(v^n, v^n, v)- \dt(e_2 \theta^n, v) = ({v^{n-1}}, v), \, \forall v \in V_1,\label{2.46a}\\
   (\theta^n, \te) + \kappa \dt((\theta^n, \te)) +\dt b_2(v^n ,\theta^n, \te) - \dt (v^n_2, \te) = ({\te^{n-1}}, \te), \, \forall \te \in V_2,\label{2.47a}
\end{gather}
and assume that 
 $\{v^n, \te^n\}$ and $\{\bar{v}^n, \bar{\te}^n\}$ are two solutions 
 corresponding
to the same initial data $\{v_0, \te_0\}\in V$.
Setting
$\tilde{v}^n=v^n-\bar{v}^n$ and $\tilde{\te}^n=\te^n-\bar{\te}^n$,
we obtain that $\{\tilde{v}^n, \tilde{\te}^n\} $ is a solution to
the following system:
\begin{gather}
(\tilde{v}^n, v)+ \nu \dt((\tilde{v}^n, v))+\dt b_1(\tilde{v}^n, {v}^n, v) +k b_1(\bar{v}^{n}, \tilde{v}^n, v)- \dt(e_2 \tilde{\theta}^n, v) = 0, \, \forall v \in V_1,\label{2.46b}\\
   (\tilde{\theta}^n, \te) + \kappa \dt((\tilde{\theta}^n, \te)) +\dt b_2(\tilde{v}^n ,\theta^n, \te) + \dt b_2(\bar{v}^{n},\tilde{\te}^n, \te)- \dt (\tilde{v}^n_2, \te) = 0. \, \forall \te \in V_2,\label{2.47b}
\end{gather}
Taking $v=\tilde{v}^n$ in \eqref{2.46b} and using \eqref{b4.1}, we
obtain
\begin{equation}\label{2.52b}
|\tilde{v}^n|^2+\nu \dt \|\tilde{v}^n\|^2+\dt b_1(\tilde{v}^n, v^n,
\tilde{v}^n) - \dt (e_2 \tilde{\te}^n, \tilde{v}^n)=0.
\end{equation}
Using property \eqref{b1.1} of the trilinear form $b_1$ and the
bound \eqref{1.87} below on $\|v^n\|$, we obtain (for $\dt \leq
\kappa_7(\|\{v_0, \te_0\}\|)$, with $\kappa_7(\|\{v_0, \te_0\}\|)$
given in Theorem \ref{thm} below):
\begin{equation}
\begin{split}
\dt b_1(\tilde{v}^n, v^n, \tilde{v}^n) &\leq c_b \dt |\tilde{v}^n|
\|\tilde{v}^n\| \|v^n\|\leq c_b K_5 \dt |\tilde{v}^n|
\|\tilde{v}^n\|\\
& \leq \frac{\nu}{4} \dt \|\tilde{v}^n\|^2+\frac{c_b}{\nu}K_5^2\dt
|\tilde{v}^n|^2.
\end{split}
\end{equation}
We also have
\begin{equation}\label{2.53b}
\begin{split}
\dt (e_2 \tilde{\te}^n, \tilde{v}^n) &\leq \dt |e_2
\tilde{\te}^n||\tilde{v}^n|\leq \dt
|\tilde{\te}^n|\|\tilde{v}^n\|\\
&\leq \frac{\nu}{4} \dt \|\tilde{v}^n\|^2+\frac{1}{\nu}\dt
|\tilde{\te}^n|^2.
\end{split}
\end{equation}
Relations \eqref{2.52b}--\eqref{2.53b} imply
\begin{equation}\label{2.54b}
\left( 1-\frac{c_b}{\nu}K_5^2\dt\right)|\tilde{v}^n|^2+\frac{\nu}{2}
\dt \|\tilde{v}^n\|^2 \leq \frac{1}{\nu}\dt |\tilde{\te}^n|^2.
\end{equation}

Now taking $\te=\tilde{\te}^n$ in \eqref{2.47b} and using
\eqref{b4.2}, we obtain
\begin{equation}\label{2.55b}
|\tilde{\te}^n|^2 +\kappa \dt\|\tilde{\te}^n\|^2+\dt
b_2(\tilde{v}^{n},  \te^n, \tilde{\te}^n)-
\dt(\tilde{v}^n_2,\tilde{\te}^n)=0.
\end{equation}
Using property \eqref{b2} of the trilinear form $b_2$ and the bound
\eqref{1.87} below on $\|\te^n\|$, we obtain
\begin{equation}\label{2.56b}
\begin{split}
\dt b_2(\tilde{v}^{n},  \te^n, \tilde{\te}^n) &\leq c_b \dt
|\tilde{v}^n|^{1/2} \|\tilde{v}^n\|^{1/2} \|\te^n\|
|\tilde{\te}^n|^{1/2} \|\tilde{\te}^n\|^{1/2}\\
&\leq \frac{\nu}{4} \dt \|\tilde{v}^n\|^2+\frac{\kappa}{4} \dt
\|\tilde{\te}^n\|^2+c K_5^2 \dt |\tilde{v}^n|^2+c K_5^2 \dt
|\tilde{\te}^n|^2.
\end{split}
\end{equation}
We also have
\begin{equation}\label{2.57b}
\begin{split}
\dt(\tilde{v}^n_2,\tilde{\te}^n)& \leq \dt
 |\tilde{v}^n_2||\tilde{\te}^n|\leq \dt
|\tilde{v}^n|\|\tilde{\te}^n\|\leq \frac{\kappa}{4} \dt
\|\tilde{\te}^n\|^2+\frac{1}{\kappa} \dt|\tilde{v}^n|^2.
\end{split}
\end{equation}
Relations \eqref{2.55b}--\eqref{2.57b} yield
\begin{equation}\label{2.58b}
(1-c K_5^2 \dt)|\tilde{\te}^n|^2+\frac{\kappa}{2} \dt
\|\tilde{\te}^n\|^2\leq\frac{\nu}{4} \dt \|\tilde{v}^n\|^2+c K_5^2
\dt |\tilde{v}^n|^2+\frac{1}{\kappa} \dt|\tilde{v}^n|^2.
\end{equation}
Adding relations \eqref{2.54b} and \eqref{2.58b}, we obtain
\begin{equation}\label{2.59b}
\begin{split}
&\left( 1-\frac{c_b}{\nu}K_5^2\dt-c K_5^2
\dt-\frac{1}{\kappa}\dt\right)|\tilde{v}^n|^2+\left(1-c K_5^2
\dt-\frac{c}{\nu}\dt \right)|\tilde{\te}^n|^2\\
&+\frac{\nu}{4} \dt\|\tilde{v}^n\|^2+ \frac{\kappa}{2} \dt
\|\tilde{\te}^n\|^2\leq 0.
\end{split}
\end{equation}
Assuming $\dt$ is sufficiently small,  that is
\begin{equation}\label{2.60b}
\dt \leq \min \left\{\kappa_7(\|\{v_0, \te_0\}\|), \frac{1}{2 \left( \frac{c_b}{\nu}K_5^2+c K_5^2
+\frac{1}{\kappa}\right)}, \frac{1}{2 \left(c K_5^2
+\frac{c}{\nu}\right)} \right\},
\end{equation}
relation \eqref{2.59b} implies $\tilde{v}^n=\tilde{\te}^n=0$. Hence, the system \eqref{2.45a}--\eqref{1.26} possesses a unique solution, provided that the time-step satisfies the constraint \eqref{2.60b}. This is enough to uniquely define the sequence $\{v^n, \te^n\}$ for $\dt$ small enough, but the dependence of the time step $k$ on the
initial data prevents us from defining a single-valued attractor in
the classical sense, and this is why we need  the
theory of the multi-valued attractors, that we discuss in Subsection 5.1.

Our next aims are to prove that the solution  $u^n=\{v^n, \te^n\}$
to the discrete system \eqref{2.45a}--\eqref{1.26} is uniformly
bounded in the $V$-norm and then to show that the global attractors
generated by the numerical scheme \eqref{2.45a}--\eqref{1.26}
converge to the global attractor of the continuous system as the
time-step  approaches zero.

In this article we only consider time discretization, we do not consider space discretization.
Important background information on space discretization and on various computational methods
can be found in some of the books and articles available in the literature. On finite elements, see, e.g., \cite{Gir}, \cite{Hey};
on finite differences and finite elements, \cite{MT}, \cite{temam:nse};
on spectral methods, \cite{Ber}, \cite{Got}.


\medskip


\section{$H$-Uniform Boundedness of $v^n$ and $\te^n$}\label{s3:reg}
In proving the $H$-uniform boundedness of $v^n$ and $\te^n$, we need first to prove a variant of the maximum principle for $\te^n$. In order to do so, we introduce the following truncation operators (cf. \cite{temam:iddsmp}), that associate with the function $\varphi $, the functions $\varphi _+$ and $\varphi _-$, given by
\begin{equation}\label{trunc}
\varphi _+(x)=\max(\varphi (x), 0), \quad \varphi _-(x)=\max(-\varphi (x),0).
\end{equation}
Note that, with this notation, we have $\varphi=\varphi _+ - \varphi _-$, the absolute value $|\varphi|$ of $\varphi$ is $\varphi _+ + \varphi _-$ and $\varphi _+  \varphi _-=0$.
Using these operators, we can prove the following preliminary lemma

\begin{Lem}
If $\varphi , \psi \in L^2(\Omega)$, then
\begin{equation}\label{1.26a}
2(\varphi -\psi, \varphi _+) \geq |\varphi _+|^2 - |\psi_+|^2+|\varphi _+-\psi_+|^2,
\end{equation}
\begin{equation}\label{1.26b}
-2(\varphi -\psi, \varphi _-) \geq |\varphi _-|^2 - |\psi_-|^2+|\varphi _--\psi_-|^2.
\end{equation}
\end{Lem}
\begin{proof}
We have
\begin{equation}
\begin{split}
2(\varphi &-\psi, \varphi _+)=2(\varphi _+ - \varphi _- -\psi_++\psi_-, \varphi _+)\\
&=2(\varphi _+ -\psi_+, \varphi _+)-2(\varphi _- -\psi_-, \varphi _+)\\
&=|\varphi _+|^2 - |\psi_+|^2+|\varphi _+-\psi_+|^2 + 2 \int_{\Omega} \psi_- \varphi _+ \, dx\\
& \geq |\varphi _+|^2 - |\psi_+|^2+|\varphi _+-\psi_+|^2,
\end{split}
\end{equation}
since $\psi_- \varphi _+ \geq 0$.
The proof is similar for \eqref{1.26b} and the lemma is proved.
\end{proof}

We are now able to prove the following variant of the maximum principle for $\te^n$:
\begin{Lem}\label{max}
If $v^n$ and $\te^n$ satisfy \eqref{1.25} and \eqref{1.26}, then
\begin{gather}
 \theta^n={\tilde{\te}}^n+\bar{\te}^n,\label{1.27}
\end{gather}
with
\begin{gather}
x_2-1 \leq \tilde{\theta}^n\leq x_2, \label{1.28}\\
|\bar{\te}^n| \leq \left( |\te_+^0| + |\te_-^0|\right) (1+2 \kappa \dt)^{-\frac{n}{2}}. \label{1.28}
\end{gather}
Moreover, there exists $M_1=M_1(|\te_0|)$, given in \eqref{1.45} below, such that
\begin{equation}\label{M_1}
|\te^n| \leq M_1, \forall n \geq 1.
\end{equation}
\end{Lem}
\begin{proof}
Rewriting \eqref{1.26} in terms of $T^n=\te^n+T_0-x_2$, we find:
\begin{gather}
 \frac{1}{\dt}(T^n-T^{n-1}, T) + \kappa ((T^n, T))+ b_2(v^n ,T^n, T)= 0, \, \forall T \in V_2,  = 0, \, n \geq 1. \label{1.29}
\end{gather}
Replacing $T$ by $2 \dt (T^n-T_0)_+$ in the above equation and using \eqref{1.26a}, we obtain:
\begin{gather}
\begin{split}
|(T^n-T_0)_+|^2&-|(T^{n-1}-T_0)_+|^2\\+|(T^n-T_0)_+-(T^{n-1}-T_0)_+|^2
& +2 \dt \kappa \|(T^n-T_0)_+\|^2 \leq 0. \label{1.30}
\end{split}
\end{gather}
Using the Poincar\'e inequality \eqref{Poin}, we find
\begin{gather}
|(T^n-T_0)_+|^2 \leq \frac{1}{\alpha} |(T^{n-1}-T_0)_+|^2, \label{1.31}
\end{gather}
where
\begin{equation}\label{1.32}
      \alpha = 1+ 2 \kappa {\dt}.
\end{equation}
Using recursively \eqref{1.31}, we find
\begin{gather}
|(T^n-T_0)_+|^2 \leq (1+2 \kappa {\dt})^{-n} |(T^0-T_0)_+|^2. \label{1.33}
\end{gather}
Similarly, using \eqref{1.26b}, we obtain
\begin{gather}
|(T^n-T_1)_-|^2 \leq (1+2 \kappa {\dt})^{-n} |(T^0-T_1)_-|^2. \label{1.34}
\end{gather}
Setting
\begin{gather}
T^n=\tilde{T}^n+\bar{T}^n, \text{ with } \bar{T}^n=(T^n-T_0)_+-(T^n-T_1)_-, \label{1.35}
\end{gather}
we find that $\tilde{T}^n=T^n-(T^n-T_0)_++(T^n-T_1)_-$, so that $\tilde{T}^n=T_1$, for $T^n \leq T_1$, $\tilde{T}^n=T^n$, for $T_1 \leq T^n \leq T_0$, and $\tilde{T}^n=T_0$, for $T^n > T_0$; in all cases
\begin{gather}
T_1 \leq \tilde{T}^n \leq T_0.\label{1.36}
\end{gather}
Rewriting \eqref{1.33}--\eqref{1.35} in terms of $\te$, we obtain
\begin{gather}
|(\te^n-x_2)_+|^2 \leq (1+2 \kappa {\dt})^{-n} |(\te^0-x_2)_+|^2, \label{1.37}\\
|(\te^n-x_2+1)_-|^2 \leq (1+2 \kappa {\dt})^{-n} |(\te^0-x_2+1)_-|^2,\label{1.38}\\
\te^n+T_0-x_2=\tilde{T}^n+(\te^n-x_2)_+-(\te^n-x_2+1)_-. \label{1.39}
\end{gather}
Setting
\begin{gather}
\bar{\te}^n=(\te^n-x_2)_+-(\te^n-x_2+1)_-,\label{1.41}\\
\tilde{\te}^n =\tilde{T}^n -T_0+x_2, \label{1.42}
\end{gather}
equation \eqref{1.39} becomes
\begin{gather}
\te^n=\tilde{\te}^n+\bar{\te}^n.  \label{1.40}
\end{gather}
By \eqref{1.36}, we have
\begin{gather}
x_2-1 \leq \tilde{\te}^n \leq x_2,\label{1.43}
\end{gather}
and by \eqref{1.41}, \eqref{1.37} and \eqref{1.38} we derive
\begin{gather}
\begin{split}
|\bar{\te}^n|&\leq |(\te^n-x_2)_+|+|(\te^n-x_2+1)_-|\\
&\leq (1+2 \kappa {\dt})^{-\frac{n}{2}}(|\te_+^0| + |\te_-^0|).\label{1.44}
\end{split}
\end{gather}
To complete the proof of the lemma, we note that \eqref{1.40}, \eqref{1.43} and \eqref{1.44} yield
\begin{equation}\label{1.44a}
|\te^n| \leq |\Omega|^{1/2} + \left( |\te_+^0| + |\te_-^0|\right) (1+2 \kappa \dt)^{-\frac{n}{2}}, \forall n \geq 1,
\end{equation}
and 
setting
\begin{equation}\label{1.45}
M_1(|\te_0|)=|\Omega|^{1/2}+|\te_+^0| + |\te_-^0|,
\end{equation}
we obtain conclusion \eqref{M_1} of the lemma.
\end{proof}

\begin{Cor}\label{C1}
If
\begin{equation}\label{3.27a}
\dt \leq \frac{1}{2\kappa},
\end{equation}
then $B_{L^2}(0, 2|\Omega|^{1/2})$, the ball in $L^2$ centered at $0$ and
radius $2|\Omega|^{1/2}$, is an absorbing ball for $\te^n$ in $L^2$.
\end{Cor}
\begin{proof}
Indeed, let $\B$ be any bounded set in $L^2$ and assume that it is
included in a ball $B(0,R)$ of $L^2$. It is easy to deduce from
\eqref{1.44a} that  for any $\te_0 \in B(0,R)$,
\begin{equation}\label{1.45b}
|\te^n| \leq |\Omega|^{1/2} + 2R (1+2 \kappa \dt)^{-\frac{n}{2}}, \forall n \geq 1,
\end{equation}
and using assumption \eqref{3.27a} on $\dt$ and the fact that $1+x \geq \exp(x/2) \text{ if } x \in (0, 1)$, we
obtain that there exists
$N_0^1(R,\dt):=\frac{ 2\ln
\left(\frac{2R}{|\Omega|^{1/2}}\right)}{\kappa \dt}$ such that
$\te^n \in B_{L^2}(0, 2|\Omega|^{1/2}), \forall n \geq N_0^1$. This completes the proof of the corollary.
\end{proof}


We are now able to prove the $H$-uniform boundedness of $v^n$. More precisely, we have the following:

\begin{Lem}\label{t:bdh}
Let  $\{v^n,\te^n\}$ be the solution of the numerical
scheme \eqref{1.25}--\eqref{1.26}. Then for every ${\dt}>0$, we have
\beq\label{q:bdv}
   |v^n|^2 \le \left(1+ \nu {\dt} \right)^{-n}|v_0|^2
                        + \frac{M_1^2 }{ \nu^2}
                         \left[1- \left(1+\nu
                        {\dt} \right)^{-n} \right],
    \>\forall\, n\ge0.
\eeq

Moreover, there exists $K_1=K_1(|{v}_0|,|\te_0|)$, such
that \beq\label{q:bdinh}
    |{v}^n| \leq K_1, \quad \forall \, n \geq 0,
\eeq
and
\beq\label{q:vbdinth}
\nu {\dt} \sum_{j=i}^{m} \|{v}^j\|^2
      \leq |{v}^{i-1}|^2 +\frac{1}{\nu}\dt \sum_{j=i}^{m}|\te^j|^2,  \,\quad\forall \, i=1,
\cdots,m,
 \eeq
\beq\label{q:tebdinth}
  \kappa {\dt} \sum_{j=i}^{m} \|\te^j\|^2\leq |\te^{i-1}|^2 +
             \frac{1}{\kappa}\dt \sum_{j=i}^{m}|v^{j}|^2, \quad\forall \, i=1, \cdots, m.
\end{equation}
\end{Lem}

\begin{proof}
Taking $v$ to be $2 {\dt} v^n$ in \eqref{1.25} 
and using the relation
\begin{equation}
      2(\varphi - \psi, \varphi)=|\varphi|^2-|\psi|^2+|\varphi-\psi|^2,
\eeq
as well as the skew property \eqref{b4.1}, we obtain
\begin{equation}\label{1.46}
\begin{split}
      |v^n|^2 - |v^{n-1}|^2 + |v^n - v^{n-1}|^2
         + 2 \nu {\dt}\, \|v^n\|^2 = 2  {\dt} (e_2 \te^n,v^n).
\end{split}
\end{equation}
Using the Cauchy--Schwarz inequality and the Poincar\'e inequality
\eqref{Poin}, we
majorize the right-hand side of \eqref{1.46} by
\begin{equation}\label{1.47}
\begin{split}
      2  {\dt} (e_2 \te^n,v^n)
       &\leq 2  {\dt} |e_2 \te^n||v^n|
       \leq 2  {\dt} | \te^n||v^n|\\
       &\leq 2  {\dt} | \te^n| \|v^n\| \leq  \nu {\dt} \|v^n\|^2 + \frac{1}{\nu} \dt\, |\te^n|^2.
\end{split}
\end{equation}
Relations \eqref{1.46} and \eqref{1.47} imply
\begin{equation}\label{1.48}
       |v^n|^2 - |v^{n-1}|^2 + |v^n - v^{n-1}|^2
         +  \nu {\dt}\, \|v^n\|^2
          \leq  \frac{1}{\nu} \dt\, |\te^n|^2.
\end{equation}
Using again the Poincar\'e inequality \eqref{Poin}, we find
\begin{equation}\label{1.49}
      |v^n|^2 \leq \frac{1}{\alpha}|v^{n-1}|^2
         + \frac{1}{\al \nu} \dt \, |\te^n|^2,
\end{equation}
where
\begin{equation}\label{1.50}
      \alpha= 1+ \nu \dt.
\end{equation}
Using recursively \eqref{1.49}, we find
\begin{equation}\label{1.51}
\begin{split}
      |v^n|^2 &\leq \frac{1}{\alpha^n}|v^0|^2 + \frac{1}{\nu}\dt
                        \sum_{i=1}^{n}\frac{1}{\alpha^i}|\te^{n+1-i}|^2\\
                  &\leq \left(1+ \nu {\dt} \right)^{-n}|v^0|^2+ \frac{M_1^2 }{ \nu^2}
                         \left[1- \left(1+\nu{\dt} \right)^{-n} \right],
\end{split}
\end{equation}
which proves \eqref{q:bdv}. 
%

Taking $K_1^2=|v^0|^2+ \frac{M_1^2 }{ \nu^2}$ relation
\eqref{q:bdinh} follows right away.

Adding inequalities \eqref{1.48} with $n$ from $i$ to $m$ we obtain
\eqref{q:vbdinth}.


Now, replacing $\te$ by $2 {\dt} \te^n$ in \eqref{1.26} and using the skew property \eqref{b4.2},
we obtain
\begin{equation}\label{1.52}
\begin{split}
      |\te^n|^2 - |\te^{n-1}|^2 + |\te^n - \te^{n-1}|^2
         + 2 \kappa {\dt}\, \|\te^n\|^2 = 2  {\dt} (v_2^n,\te^n).
\end{split}
\end{equation}
Using again the Cauchy--Schwarz inequality and the Poincar\'e inequality
\eqref{Poin}, we
majorize the right-hand side of \eqref{1.52} by
\begin{equation}\label{1.53}
\begin{split}
      2  {\dt} (v_2^n,\te^n)
       &\leq 2  {\dt} |v_2^n||\te^n|\leq 2  {\dt} |v^n|\|\te^n\|\\
       & \leq  \kappa {\dt} \|\te^n\|^2 + \frac{1}{\kappa} \dt\, |v^n|^2.
\end{split}
\end{equation}
Relations \eqref{1.52} and \eqref{1.53} imply
\begin{equation}\label{1.54}
       |\te^n|^2 - |\te^{n-1}|^2 + |\te^n - \te^{n-1}|^2
         +  \kappa {\dt}\, \|\te^n\|^2
          \leq  \frac{1}{\kappa} \dt\, |v^n|^2.
\end{equation}
Summing inequalities \eqref{1.54} with $n$ from $i$ to $m$
we obtain \eqref{q:tebdinth}.
\end{proof}


\begin{Cor}\label{C2}
Let
\begin{equation}\label{1.45a}
\dt \leq \min\left\{\frac{1}{2\kappa}, \frac{1}{\nu}\right\}=:\kappa_1,
\end{equation}
and set $\rho_0=2|\Omega|^{1/2}+\frac{\sqrt 5 |\Omega|^{1/2}}{\nu}$.
Then $B_{H}(0, \rho_0)$, the ball in $H$ centered at $0$ and radius
$\rho_0$, is an absorbing ball for $\{{v}^n, \te^n\}$ in $H$.
\end{Cor}
\begin{proof}
Let $\B$ be any bounded set in $H$ and assume that it is
included in a ball $B(0,R)$ of $H$. For any initial data
$\{{v}^0, \te^0\} \in \B$, Corollary \ref{C1} implies that
\begin{equation}
|\te^n| <  2|\Omega|^{1/2},  \forall n \geq N_0^1(R,\dt),
\end{equation}
and then \eqref{1.49} becomes
\begin{equation}\label{}
      |v^n|^2 \leq \frac{1}{\alpha}|v^{n-1}|^2
         + \frac{4 }{\al \nu}|\Omega| \dt, \, \forall n \geq N_0^1(R,\dt),
\end{equation}
where
\begin{equation}\label{}
      \alpha= 1+ \nu \dt.
\end{equation}
Iterating the above inequality, we find (for any $n \geq N_0^1(R,\dt)$)
\begin{equation}\label{1.51}
\begin{split}
      |v^n|^2 &\leq \frac{1}{\alpha^{(n-N_0^1)}}|v^{N_0^1}|^2 + \frac{4}{\nu}|\Omega|\dt
                        \sum_{i=1}^{n-N_0^1}\frac{1}{\alpha^i}\\
                  &= \left(1+ \nu {\dt} \right)^{-(n-N_0^1)}|v^{N_0^1}|^2+ \frac{4}{\nu^2}|\Omega|
                         \left[1- \left(1+\nu{\dt} \right)^{-(n-N_0^1)} \right],\\
                         & \leq  \left(1+ \nu {\dt} \right)^{-(n-N_0^1)}\left[R^2+\frac{4}{\nu^2}(|\Omega|+2R^2)\right]+ \frac{4}{\nu^2}|\Omega|\\
                         & \qquad (\text{by } \eqref{q:bdv} \text{ and } \eqref{1.45}),
\end{split}
\end{equation}
and using assumption \eqref{1.45a} on $\dt$ and the fact that $1+x \geq \exp(x/2) \text{ if } x \in (0, 1)$, we
obtain that there exists $N_0^2(R,\dt)$,
\begin{equation}
N_0^2(R,\dt):=\frac{2}{\nu \dt}\ln \frac{\nu^2\left[R^2+\frac{4}{\nu^2}(|\Omega|+2R^2)\right]}{|\Omega|},
\end{equation}
such that
$|v^n| \leq \sqrt 5 |\Omega|^{1/2}/\nu$, $\forall n \geq N_0^1+N_0^2=:N_0(R,\dt).$ 

We, therefore, have that $\{v^n, \te^n\} \in B_{H}(0,
\rho_0)$, for all $n \geq N_0(R, \dt)$, which completes the proof of
the corollary.
\end{proof}

\section{$V$-Uniform Boundedness of $v^n$ and $\te^n$}\label{s4}

We now seek to obtain uniform bounds for $v^n$ and $\te^n$ in $V$, similar to those we have already obtained in $H$ (see \eqref{q:bdinh} and \eqref{M_1} above). In order to do this, we first derive bounds for $v^n$ valid on any finite interval of time (see Proposition \ref{t:expbd} below), and then we repeatedly use them on successive intervals of time together with (a discrete uniform Gronwall) Lemma \ref{t:dugronwall} to arrive at the desired uniform bounds. Once we have obtained the $V$-uniform bounds on $v^n$, we can use those, together with a new version of the discrete uniform Gronwall lemma, to derive the $V$-uniform boundedness of $\te^n$.


\subsection{$H^1$-Uniform Boundedness of $v^n$}\label{ss: 4.1}

\begin{Lem}\label{l2}
For every ${\dt}>0$, we have
\beq\label{1.66}
   \|v^n\|^2 \leq K_2 \|v^{n-1}\|^2 + \frac{4}{\nu^2} M_1^2, \, \forall n\geq 1,
\eeq
where $K_2=2(1 +2 c_b^2 K_1^2 /\nu^2 )$.
\end{Lem}

\begin{proof}
Replacing $v$ by $2\dt (v^n-v^{n-1})$ in  \eqref{1.25},
we obtain
\beq\label{1.62}
\begin{split}
  2 |v^n - v^{n-1}|^2
  &+ \nu \dt \|v^n\|^2 - \nu \dt \|v^{n-1}\|^2
    + \nu \dt \|v^n - v^{n-1}\|^2 \\
  &+ 2 \dt\, b_1(v^n,v^n,v^n-v^{n-1})
  = 2 \dt\, (e_2\te^n,v^n - v^{n-1}).
\end{split}
\eeq
Using properties \eqref{b4.1}, \eqref{b5.1} and \eqref{b1.1} of the trilinear form $b_1$
and recalling \eqref{q:bdinh}, we bound the nonlinear term as
\beq\label{1.63}
\begin{split}
 2\dt b_1(v^n,v^n,v^n-v^{n-1}) &= 2 \dt b_1(v^n,v^{n-1},v^n) \quad (\text{by \eqref{b4.1}, \eqref{b5.1}})\\
  &\le 2 c_b \dt |v^n| \|v^n\| \|v^{n-1}\| \quad (\text{by \eqref{b1.1}})\\
  &\le \frac{\nu}{2} \dt \|v^n\|^2
   + \frac{2 c_b^2}{\nu} K_1^2 \dt \|v^{n-1}\|^2.
\end{split}
\eeq
We bound the right-hand side of \eqref{1.62} using Cauchy--Schwarz'
inequality, \eqref{Poin} and \eqref{M_1}:
\beq\label{1.64}
\begin{split}
 2 {\dt} (e_2 \te^n, v^n-v^{n-1})
     &\leq 2 {\dt} |\te^n| |v^n-v^{n-1}|\\
     &\leq  {\dt} |\te^n| \|v^n-v^{n-1}\| \\
     &\leq \frac{\nu}{2} \dt \|v^n-v^{n-1}\|^2 + \frac{2}{\nu}\dt M_1^2.
\end{split}\eeq
Gathering relations \eqref{1.62} through \eqref{1.64}, we find
\beq\label{1.65}
\begin{split}
  2 &|v^n - v^{n-1}|^2 +  \frac{\nu}{2} \dt \|v^n\|^2
 - \Bigl(\nu+\frac{2 c_b^2}{\nu} K_1^2 \Bigr) \dt \|v^{n-1}\|^2\\
  &+ \frac{\nu}{2} \dt\, \|v^n - v^{n-1}\|^2
    \leq \frac{2}{\nu} \dt M_1^2,
\end{split}
\eeq
We thus  obtain
\beq
   \|v^n\|^2 \leq K_2 \|v^{n-1}\|^2 + \frac{4}{\nu^2} M_1^2,
\eeq
which is exactly conclusion \eqref{1.66} of the lemma.
\end{proof}

\begin{Lem}\label{l3}
For every ${\dt}>0$, we have
\begin{equation}\label{1.58}
  c_1 K_1^2 {\dt}\|v^n\|^4- \|v^n\|^2
     + \|v^{n-1}\|^2 + \frac{2}{\nu} \dt M_1^2 \geq 0,  \, \forall n \geq 1,
\eeq
where $c_1=27 c_b^4/(2 \nu^3)$.

\end{Lem}
\begin{proof}
Replacing $v$ by $2 {\dt} A_1 v^n$ in  \eqref{1.25},
we obtain
\beq\label{q:uhnquad}
\begin{split}
    \|v^n\|^2 - \|v^{n-1}\|^2
     & + \|v^n - v^{n-1}\|^2 +2 \dt b_1(v^n, v^n, A_1 v^n)\\
     & + 2 \nu {\dt} |A_1 v^n|^2  = 2 {\dt} (e_2 \te^n, A_1 v^n).
\end{split}
\eeq
Using property \eqref{b2.1} of the trilinear form $b_1$ and
recalling \eqref{q:bdinh}, we have the following bound of the
nonlinear term,
\beq\label{1.55}
\begin{split}
 2 \dt b_1(v^n, v^n,A_1 v^n)
  &\leq 2\, c_b\, \dt\, |v^n|^{1/2}  \|v^n\| |A_1 v^n|^{3/2}\\
  &\leq \frac{\nu}{2} \dt | A_1 v^n|^2 + \frac{27 c_b^4}{2 \nu^3} K_1^2 \dt \|v^n\|^4.
\end{split}\eeq
Using the Cauchy--Schwarz inequality and recalling \eqref{M_1}, we bound the right-hand side of~\eqref{q:uhnquad} by
\beq\label{1.56}\begin{split}
 2 {\dt} (e_2 \te^n, A_1 v^n)
     &\leq 2 {\dt} |\te^n| | A_1v^n|\\
     &\leq \frac{\nu}{2} \dt | A_1 v^n|^2 + \frac{2}{\nu} \dt M_1^2.
\end{split}\eeq

Relations \eqref{q:uhnquad}--\eqref{1.56} imply
\beq\label{1.57}
\begin{split}
    \|v^n\|^2 - \|v^{n-1}\|^2
      & + \|v^n - v^{n-1}\|^2 + \nu {\dt} |A_1 v^n|^2 \\
      &  \leq \frac{27 c_b^4}{2 \nu^3} K_1^2 \dt \|v^n\|^4 +  \frac{2}{\nu} \dt M_1^2 ,
\end{split}
\eeq
from which we obtain conclusion \eqref{1.58} of Lemma \ref{l3}.
\end{proof}
In what follows, we will make use of the following discrete Gronwall lemma, whose proof can
be found in \cite{shen:90}, \cite{TW}.

\begin{Lem}\label{t:dgronwall}
Given $\dt>0$, an integer $n_*>0$ and positive sequences $\x_n$,
$\g_n$ and $\h_n$ such that
\begin{equation}\label{q:gronseq}
  \x_n \le \x_{n-1} (1 + \dt \g_{n-1}) + \dt \h_n,
    \qquad\textrm{for}\> n = 1, \cdots, n_*,
\end{equation}
we have, for any $n \in \{2, \cdots, n_*\}$,
\begin{equation}\label{q:gronest}
  \x_n \le \x_0 \exp\biggl( \dt \sum_{i=0}^{n-1} \g_i \biggr)
    + \dt \sum_{i=1}^{n-1}  \h_i \exp\biggl(  \dt \sum_{j=i}^{n-1} \g_j \biggr)
    + \dt \h_n \,.
\eeq
\end{Lem}

\begin{Prop}[Estimates on a finite interval]\label{t:expbd}
Let $T>0$ be fixed and let $K_3(\cdot,\cdot,\cdot)$ be the function,
monotonically increasing in all its arguments, given in \eqref{K3}
below. 
If the timestep $\dt$ is such that
\begin{equation}\label{q:ksmall}
      \dt \leq \min \{ \kappa_1, \kappa_2(|v_0|,|\te_0|), \kappa_3(\|v^0\|,|\te_0|,T)\},
\eeq
where $\kappa_1$ is given by \eqref{1.45a},
\begin{align}
    \kappa_2(|v_0|, |\te_0|) &= \frac{\nu^2}{40 c_1 K_1^2 M_1^2},
    \label{k2}\\
    \kappa_3(\|v^0\|,|\te_0|,T) &= \frac{1}{10 c_1 K_1^2 K_2 K_3^2(\|v^0\|,|\te_0|,T)},
        \label{k3}
\end{align}
then

(i)
\begin{equation*}\label{q:bdinv}
   \|v^n\| \le K_3\bigl(\|v^0\|, |\te_0|, n\dt \bigr),
      \, \forall\, n = 0, \cdots, N:= \lfloor T/\dt \rfloor,
\end{equation*}
(ii)
\begin{equation*}\label{q:dt-unh}
\begin{split}
  c_1 K_1^2 {\dt} \left( K_2 \|v^{n}\|^2 + + \frac{4}{\nu^2} M_1^2\right)
    \leq \frac{1}{5}, \, \forall\, n = 0, \cdots, N:= \lfloor T/\dt \rfloor,
\end{split}
\end{equation*}
(iii)
\begin{equation*}
\begin{split}
  \|v^n\|^2 \leq \|v^{n-1}\|^2
          &\left[ 1 + 2 c_1 K_1^2 {\dt} \left(
             \|v^{n-1}\|^2 +\frac{2}{\nu}  {\dt} M_1^2 \right) \right]
          + \frac{14}{5\nu} {\dt} M_1^2, \\
         &  \forall\, n = 1, \cdots, N:= \lfloor T/\dt \rfloor.
\end{split}
\end{equation*}
\end{Prop}

\begin{proof}
We will use induction on $n$. 
The induction consists in showing that (i) holds at $n=0$, and then showing that if (i) holds for $ n \leq m-1$,  then (i) holds for $n=m$.
If $n=0$, one can easily see using the definition \eqref{K3} of $K_3$ that (i) is true.
Now assume that (i) holds for $ n \leq m-1$. Then, by \eqref{q:ksmall},
we have that (ii) is true for $ n \leq m-1$ and
\begin{equation}\label{1.61}\begin{split}
      \Delta_{n}
        &= 1-4 c_1 K_1^2 {\dt} \left(\|v^{n}\|^2 + \frac{2}{\nu}  {\dt} M_1^2\right)> 0, \, \forall \, n \leq m-1.
\end{split}\eeq
From \eqref{1.58}, we obtain that either
\begin{equation}\label{1.59}
      \|v^{n+1}\|^2 \leq \frac{1-\sqrt{\Delta_{n}\vphantom{|}}}{2 c_1 K_1^2 \dt },
 \eeq
or
\begin{equation}\label{1.60}
      \|v^{n+1}\|^2 \geq \frac{1+\sqrt{\Delta_{n}\vphantom{|}}}{2 c_1 K_1^2 \dt }.
 \eeq
The second alternative is not possible, as it implies, with \eqref{1.66}, (i) at order $n$ and \eqref{q:ksmall}:
\begin{equation}
\begin{split}
1 & \leq 1+\sqrt{\Delta_{n}} \leq 2 c_1 K_1^2 \dt \|v^{n+1}\|^2 \\
&\leq 2 c_1 K_1^2 \dt \left(K_2 \|v^{n}\|^2 + \frac{4}{\nu^2} M_1^2 \right) \leq \frac{2}{5}.
\end{split}
\end{equation}
Then the first alternative holds and, thus,
\begin{equation}
2 c_1 K_1^2 \dt \|v^{n+1}\|^2 \leq 1-\sqrt{1-x},
\end{equation}
with
\[
x=4 c_1 K_1^2 {\dt} \left(\|v^{n}\|^2 + \frac{2}{\nu}  {\dt} M_1^2\right).
\]
Since
\[
1-\sqrt{1-x}=\frac{x}{1+\sqrt{1-x}} \leq \frac{x}{2} \left(1+\frac{x}{2} \right), \quad 0 \leq x \leq \frac{4}{5},
\]
we obtain
\begin{equation}
\begin{split}
2 c_1 K_1^2 \dt \|v^{n+1}\|^2 \leq 2 &c_1 K_1^2 {\dt} \left(\|v^{n}\|^2 + \frac{2}{\nu}  {\dt} M_1^2\right)\\
&\cdot \left[1+ 2 c_1 K_1^2 {\dt}\left(\|v^{n}\|^2 + \frac{2}{\nu}  {\dt} M_1^2\right)\right],
\end{split}
\end{equation}
which implies, together with (ii) at order $n$:
\beq\label{1.69}
\begin{split}
  \|v^{n+1}\|^2 \leq \|v^{n}\|^2
          \left[ 1 + 2 c_1 K_1^2 {\dt} \left(
            \|v^{n}\|^2 +\frac{2}{\nu}  {\dt} M_1^2 \right) \right]
          + \frac{14}{5\nu} {\dt} M_1^2,
\end{split}
\eeq
which is exactly (iii) at order $n+1$. 
We rewrite (iii) in the form
\beq\label{4.30}
 \x_n \leq \x_{n-1} (1+ \dt \g_{n-1}) + \dt \h,
\eeq
with
\begin{equation}\label{q:expgr}
   \x_n = \|v^n\|^2,
   \quad \g_n = 2 c_1 K_1^2 \left(\|v^{n}\|^2 + \frac{2}{\nu}{\dt}M_1^2\right)
   \quad\text{and}\quad \h=\frac{14}{5\nu} M_1^2,
\end{equation}
and we have that \eqref{4.30} holds for $n=1, \cdots, m$. In order to prove that (i) holds for $n=m$,  we use the (discrete Gronwall) Lemma \ref{t:dgronwall} and we compute the following. If $i>0$, using \eqref{q:vbdinth}, we obtain:

\begin{equation}\label{1.70ab}
\begin{aligned}
   \sum_{j=i}^{m-1} \dt\eta_j &= 2 c_1 K_1^2 {\dt} \sum_{j=i}^{m-1}
         \left(\|v^{j}\|^2 + \frac{2}{\nu}{\dt}M_1^2\right)\\
    &\le \frac{2}{\nu} c_1 K_1^2\left[ K_1^2 + \frac{3}{\nu}M_1^2 (m-i) {\dt}\right].
\end{aligned}
\end{equation}
Similarly, if $i=0$, using again \eqref{q:vbdinth} and recalling \eqref{q:ksmall} and \eqref{k3}, we find
\begin{equation}\label{1.71}
\begin{aligned}
   \sum_{j=0}^{m-1} \dt\eta_j &= 2 c_1 K_1^2 {\dt}  \sum_{j=0}^{m-1}
        \left(\|v^{j}\|^2 + \frac{2}{\nu}{\dt}M_1^2\right)\\
    &\le \frac{1}{5}+\frac{2}{\nu} c_1 K_1^2\left( K_1^2 + \frac{3}{\nu}M_1^2 m {\dt}\right).
\end{aligned}
\end{equation}
Using \eqref{1.70ab} we also have
\begin{equation}\label{1.72}
\begin{aligned}
    \sum_{i=1}^{m-1} \,\dt\h\,\exp\biggl(\, \sum_{j=i}^{m-1}\, \dt\g_j \biggr)
    &\le \frac{14}{5\nu} M_1^2 \exp\left\{\frac{2}{\nu} c_1 K_1^2\left( K_1^2 + \frac{3}{\nu}M_1^2 m {\dt}\right)\right\} m \dt.
\end{aligned}
\end{equation}
By \eqref{q:gronest}, \eqref{1.71}, and \eqref{1.72} we obtain
\begin{equation}\label{K3}
\begin{aligned}
\|v^m\|^2 &\leq\left(\|v^0\|^2+\frac{28}{5\nu} M_1^2 m \dt \right)\exp\left\{ \frac{1}{5}+\frac{2}{\nu} c_1 K_1^2\left(K_1^2 + \frac{3}{\nu}M_1^2 m {\dt}\right) \right\}\\
    &=: K_3^2(\|v^0\|, |\te_0|, m\dt),
\end{aligned}
\end{equation}
which is exactly (i) for $n=m$.

We  note that the dependence of $K_3$,  $\kappa_3$ and $\kappa_2$ on $|v_0|$ and $|\te_0|$ is through $K_1$ and $M_1$,
but those quantities bound all $|v^n|$ and $|\theta^n|$, respectively.

This completes the proof of the proposition.
\end{proof}
We now want to extend Proposition \ref{t:expbd} to obtain uniform bounds for $\|v^n\|$, for all $n \geq 0$. In order to do so, we will repeatedly apply Proposition \ref{t:expbd} on finite intervals of time, considering different initial values. Since the bound $K_3$ obtained on each
interval is an increasing function in the corresponding initial value considered (see \eqref{K3}), bounding uniformly all initial values will provide a uniform bound for all $\|v^n\|$,  $n \geq 0$. To do so,  we need the following (discrete
uniform Gronwall) lemma, whose proof can be found in \cite{TW} (see
also \cite{shen:90}).

\begin{Lem}\label{t:dugronwall}
We are given $\dt>0$, positive integers $ n_1, n_2, n_*$ such that
$n_1 < n_*$, $n_1 + n_2 + 1 \leq  n_* $, and positive sequences
$\x_n$, $\g_n$, $\h_n$ such that
\begin{equation}\label{q:gronseq2}
       \x_n \le \x_{n-1} (1 + \dt \g_{n-1}) + \dt \h_n,
          \qquad\textrm{for } n = n_1, \cdots, n_*.
\end{equation}
Assume also that
\begin{equation}\label{q:groncond}
\begin{aligned}
  \dt \sum_{n=n'}^{n'+n_2} \g_n &\le a_1(n_1,n_*), \qquad\qquad
  \dt \sum_{n=n'}^{n'+n_2}  \h_n \le a_2(n_1,n_*),\\
  \dt \sum_{n=n'}^{n'+n_2}  \x_n &\le a_3(n_1,n_*),
\end{aligned}
\end{equation}
for any $n'$ satisfying $n_1 \le n' \le n_* - n_2$. We then have,
\begin{equation}\label{q:ugronest}
       \x_n \le \Bigl( \frac{a_3(n_1,n_*)}{\dt n_2}
         + a_2(n_1,n_*) \Bigr)\, {\rm e}^{a_1(n_1,n_*)},
\end{equation}
for any $n$ such that $n_1 + n_2 + 1 \le n \le n_*$.
\end{Lem}

We are now in a position to give the main result of this section, that is, to derive a
uniform bound for $\|v^n\|$, for all $n \ge 1$.

\begin{Thm}\label {mainth}
Let $v_0 \in V_1$, $\te_0 \in H_2$, 
and $\{v^n,\te^n\}$ be the solution of the numerical
scheme \eqref{1.25}--\eqref{1.26}. Also, let  
$r\ge 4 \, \kappa_1$ be
arbitrarily fixed and let $\dt$ be such that
\begin{equation}\label{q:tsmall}
\begin{split}
      \dt  &\leq \min \big\{\kappa_1, \kappa_2(|v_0|, |\te_0|),
\kappa_3(\|v_0\|,|\te_0|, T_0+r),\kappa_3(\absbalv(r),|\te_0|, r) \big\}\\
        &:=\kappa_4(\|v_0\|,|\te_0|),
\end{split}
\eeq
where $\kappa_1$  is given by \eqref{1.45a},
$\kappa_2(\cdot,\cdot)$ and $\kappa_3(\cdot,\cdot,\cdot)$ are given
in Proposition \ref{t:expbd}, $T_0$ is the time of entering the absorbing ball in $H$ for $\{v^n, \te^n\}$ and $\absbalv$ is given in
\eqref{q:absballv} below.
Then we have
\begin{equation}\label{q:unifbdv}
       \|v^n\| \leq K_4(\|v_0\|, |\te_0|), \quad \forall \, n \geq 1,
\end{equation}
where $K_4(\cdot, \cdot)$ is a continuous function defined on
$\Real_+^2$, increasing in both arguments.
\end{Thm}

\begin{proof}
In order to derive uniform bounds for $\|v^n\|$, for all $n \geq
1$, we apply Proposition \ref{t:expbd} on successive intervals of
time, with different initial values. On each interval considered, we
obtain a bound $K_3$ that depends on the norm of the initial value
(see \eqref{K3}). Using (the discrete uniform Gronwall) Lemma
\ref{t:dugronwall}, we majorize the norm of each initial value by a
constant $\rho_1$ (see \eqref{q:absballv} below) and recalling the
fact that $K_3$ is an increasing function of its arguments, we
obtain a bound independent of the initial value considered.

We start by applying Proposition \ref{t:expbd} on the interval
$[0,T_0+r]$, where $T_0:= N_0 \dt$, with $N_0$ being given in Corollary \ref{C2}. 
We obtain:
\beq\label{1.76}
\begin{split}
   \|v^n\| \le K_3\bigl(\|v^0\|,|\te_0|, T_0+r \bigr),
      & \, \forall\, n = 0, \cdots, N_0+N_r,  \\
      & \,  N_r:= \lfloor r/\dt \rfloor,
      \end{split}
\eeq
and
\beq\label{1.75}
\begin{split}
  \|v^n\|^2 \leq \|v^{n-1}\|^2
         & \left[ 1 + 2 c_1 K_1^2 {\dt} \left(
            \|v^{n-1}\|^2 +\frac{2}{\nu}  {\dt} M_1^2 \right) \right]
          + \frac{14}{5\nu} {\dt} M_1^2, \\ &\, \forall\, n = 1, \cdots, N_0+N_r.
\end{split}
\eeq

Rewriting \eqref{1.75} in the form
\[
\xi_n \leq \xi_{n-1} (1+\dt \eta_{n-1}) + \dt \zeta,
\]
with $\x_n=\|v^n\|^2$, $\g_n=2 c_1 K_1^2 \left(
            \|v^{n}\|^2 +\frac{2}{\nu}  {\dt} M_1^2 \right)$, $\h_n=\frac{14}{5\nu} M_1^2$,
we apply Lemma \ref{t:dugronwall} with $n_1=N_0+1$, $n_2=N_r-2$, and $n_*=N_0+N_r$. For
$n'=N_0+1, N_0+2$, we compute, using \eqref{q:vbdinth} and \eqref{q:tsmall} and recalling that, by Corollary \ref{C2},  the bounds $K_1$ and $M_1$ on $|v^n|$ and $|\te^n|$, respectively, can be replaced by $\rho_0$ for $n\geq N_0$:
\begin{equation*}
\begin{aligned}
 \dt \sum_{n=n'}^{n'+n_2} \g_n &= 2 c_1 \rho_0^2 \dt \sum_{n=n'}^{n'+n_2} \big(
            \|v^{n}\|^2 +\frac{2}{\nu}  {\dt} \rho_0^2 \big)\\
            &\le \frac{2}{\nu} c_1 \rho_0^2\left[ \rho_0^2 + \frac{3}{\nu}\rho_0^2  r \right],
\end{aligned}
\end{equation*}
\begin{equation*}
\begin{aligned}
 \dt \sum_{n=n'}^{n'+n_2}  \h_n = \dt \sum_{n=n'}^{n'+n_2} \frac{14}{5\nu} \rho_0^2 \le \frac{14}{5\nu} \rho_0^2 r,\\
 \dt \sum_{n=n'}^{n'+n_2}  \x_n = \dt \sum_{n=n'}^{n'+n_2} \|v^n\|^2 \le  \frac{1}{\nu} \left( \rho_0^2 + \frac{1}{\nu}\rho_0^2  r \right).
\end{aligned}
\end{equation*}
Then Lemma \ref{t:dugronwall} 
implies:
\beq\label{q:absballv}
\begin{aligned}
 \|v^{N_0+N_r}\|^2
  &\le \rho_0^2\left[ \frac{2}{\nu} \left( \frac{1}{r} + \frac{1}{\nu} \right) + \frac{14}{5\nu} r \right]\\
  &\qquad\qquad \exp \left\{ \frac{2}{\nu} c_1 \rho_0^4\left( 1 + \frac{3}{\nu}  r \right)\right\}\\
      &=: \absbalv(r)^2 \,.
\end{aligned}
\eeq
Using the above bound and recalling  assumption \eqref{q:tsmall} on $\dt$ and the fact that $\kappa_3(\cdot, \cdot, \cdot)$
is a  decreasing function of its arguments, we notice that $\dt$ satisfies the constraint \eqref{q:ksmall} in Proposition \ref{t:expbd}. We can, therefore, apply Proposition \ref{t:expbd} on the interval $[T_0+r, T_0+2r]$ with the initial data $\{v^{N_0+N_r}, \te^{N_0+N_r}\}$, to obtain that
relation (iii) holds for all $n = N_0+N_r+1,\cdots,N_0 + 2N_r$, and
\beq\label{q:bdinDA}
\begin{aligned}
   \|v^n\| \le K_3\bigl(\|v^{N_0+N_r}\|, |\te_0|, r \bigr) \leq K_3(\absbalv(r), |\te_0|,r), \\
      \, \forall\, n = N_0+N_r+1, \cdots, N_0 + 2N_r.
\end{aligned}
\eeq
Using again Lemma \ref{t:dugronwall} with $n_1= N_0+N_r+1, n_2=N_r-2$ and $n_*=N_0+2N_r$, we obtain
\begin{equation}
 \|v^{N_0+2N_r}\|^2\le\absbalv(r)^2.
\end{equation}

Iterating the above procedure, we find 
that $\|v^{N_0+iN_r}\| \leq \absbalv(r)$, for all $i=1,2, \cdots,$  and
\beq
\|v^n\| \leq K_3(\absbalv(r), |\te_0|,r), \, \forall n \geq N_0+N_r.
\eeq




Finally, recalling \eqref{1.76}, which gives a bound for $0 \leq n \leq N_0+N_r$, we obtain
\beq
\begin{split}
\|v^n\|  &\leq \max\{K_3\bigl(\|v^0\|,|\te_0|, T_0+r \bigr), K_3(\absbalv(r), |\te_0|,r)\} \\
           &=:K_4\bigl(\|v_0\|,|\te_0|), \, \forall n \geq 1.
\end{split}
\eeq

This completes the proof of the theorem.
\end{proof}

\begin{Cor}\label{4.1}
Under the assumptions of Theorem \ref {mainth}, we also have
\beq\label{4.41a}
\begin{split}
    \sum_{n=i}^{m} \|v^n - v^{n-1}\|^2  \leq &K_4^2+ \frac{27 c_b^4}{2 \nu^3} K_1^2 K_4^4 \dt (m-i+1)\\
    & +  \frac{2}{\nu}M_1^2  \dt(m-i+1), \quad\forall \, i=1, \cdots,
m.
\end{split}
\eeq
\end{Cor}

\begin{proof}
Taking the sum of \eqref{1.57} with $n$ from $i$ to $m$ and using \eqref{q:unifbdv} gives conclusion \eqref{4.41a} of the corollary right away. 
\end{proof}
\subsection{$H^1$-Uniform Boundedness of $\te^n$ }\label{ss:vte}
We are now going to prove the $H^1$-uniform boundedness of $\te^n$, for all $n \geq 0$. In order to do so, we will first use the discrete Gronwall lemma to derive an upper bound on $\|\te^n\|$, $n \leq N$, for some $N >0$, and then we will use another version of the discrete uniform Gronwall lemma (see Lemma \ref{t:dugronwall2} below) to obtain  an upper bound on $\|\te^n\|$, $n \geq N$.

\begin{Lem}\label{finite}

Let $\{v_0, \te_0\} \in V$ 
and $\{v^n,\te^n\}$ be the solution of the numerical
scheme \eqref{1.25}--\eqref{1.26}. Also, let $T>0$  be
arbitrarily fixed and $\dt$ be such that
\begin{equation}\label{t}
\begin{split}
      \dt  &\leq \min \big\{\kappa_4(\|v_0\|,|\te_0|),
\kappa_5(\|v_0\|,|\te_0|) \big\},
\end{split}
\eeq
where $\kappa_4(\cdot, \cdot)$ is given in Theorem \ref{mainth} and
\begin{equation}\label{1.77}
\kappa_5(\|v_0\|,|\te_0|)=\frac{1}{2c_2K_1^2K_4^2(\|v_0\|,|\te_0|)}, \text{ with } c_2=\frac{27 c_b^4}{32 \kappa^2}.
\eeq
Then we have
\begin{equation}\label{1.78}
\|\te^n\|^2 \leq 4^{c_2K_1^2K_4^2 T} \left( \|\te^0\|^2+\frac{2}{\kappa K_4^2}  \right), \, \forall n=1, \cdots, N:= \lfloor T/\dt \rfloor.
\eeq
\end{Lem}

\begin{proof}
Replacing $\te$ by $2 {\dt} A_2 \te^n$ in  \eqref{1.26},
we obtain
\beq\label{1.79}
\begin{split}
    \|\te^n\|^2 - \|\te^{n-1}\|^2
     & + \|\te^n - \te^{n-1}\|^2 +2 \dt b_2(v^n, \te^n,A_2 \te^n)\\
     & - 2 {\dt} (v_2^n , A_2 \te^n)+ 2 \kappa {\dt} |A_2 \te^n|^2=0.
\end{split}
\eeq
Using property \eqref{b3.2} of the trilinear form $b_2$ and
recalling \eqref{q:bdinh} and \eqref{q:unifbdv}, we have the following bound of the
nonlinear term,
\beq\label{1.80}
\begin{split}
 2 \dt b_2(v^n, \te^n,A_2 \te^n)
  &\leq 2\, c_b\, \dt\, |v^n|^{1/2}  \|v^n\|^{1/2} \|\te^n\|^{1/2} |A_2 \te^n|^{3/2}\\
  &\leq \frac{\kappa}{2} \dt | A_2 \te^n|^2 + c_2 K_1^2 K_4^2 \dt \|\te^n\|^2.
\end{split}\eeq
Using the Cauchy--Schwarz inequality and recalling \eqref{q:bdinh}, we have the following bound
\beq\label{1.81}\begin{split}
 -2 {\dt} (v_2^n , A_2 \te^n)
     &\leq 2 {\dt} |v_2^n| | A_2 \te^n|\\
     &\leq \frac{\kappa}{2} \dt | A_2 \te^n|^2 + \frac{2}{\kappa} K_1^2 \dt.
\end{split}\eeq

Relations \eqref{1.79}--\eqref{1.81} imply
\beq\label{1.82}
\begin{split}
    \|\te^n\|^2 - \|\te^{n-1}\|^2
      & + \|\te^n - \te^{n-1}\|^2 + \kappa {\dt} |A_2 \te^n|^2 \\
      &  \leq c_2 K_1^2 K_4^2\dt \|\te^n\|^2 +  \frac{2}{\kappa} K_1^2 \dt,
\end{split}
\eeq
from which we obtain
\begin{equation}\label{1.83}
      \|\te^n\|^2 \leq \frac{1}{\alpha}\|\te^{n-1}\|^2
         + \frac{2}{\kappa \alpha} K_1^2\dt ,
\end{equation}
where
\begin{equation}\label{1.84}
      \alpha= 1- c_2 K_1^2 K_4^2\dt.
\end{equation}
Using recursively \eqref{1.83}, we find
\begin{equation}\label{1.85}
      \|\te^n\|^2 \leq (1-c_2 K_1^2 K_4^2\dt)^{-n} \left( \|\te^0\|^2+\frac{2}{\kappa K_4^2}  \right).
\end{equation}
Since
\[ 1-x \geq 4^{-x}, \quad 0<x\leq\frac{1}{2},
\]
and, by hypothesis, $c_2 K_1^2 K_4^2\dt \leq 1/2$, conclusion \eqref{1.78} follows immediately.
This completes the proof of Lemma \ref{finite}.
\end{proof}

In order to derive an upper bound on $\|\te^n\|$, $n \geq N$, we will need the following version of the discrete uniform Gronwall lemma, slightly different from Lemma \ref{t:dugronwall}:

\begin{Lem}\label{t:dugronwall2}
We are given $\dt>0$, positive integers $ n_0, n_1$  and positive sequences
$\x_n$, $\g_n$, $\h_n$ such that
\begin{equation}\label{dt}
\dt \g_{n} < \frac{1}{2}, \qquad\textrm{for } n \geq n_0,
\end{equation}
\begin{equation}\label{q:gronseq2}
     (1-\dt \g_{n})  \x_n \le \x_{n-1}  + \dt \h_n,
          \qquad\textrm{for } n \geq n_0.
\end{equation}
Assume also that
\begin{equation}\label{q:groncond2}
\begin{aligned}
  \dt \sum_{n=k_0}^{k_0+n_1} \g_n &\le a_1(n_0,n_1), \qquad\qquad
  \dt \sum_{n=k_0}^{k_0+n_1}  \h_n \le a_2(n_0,n_1),\\
  \dt \sum_{n=k_0}^{k_0+n_1}  \x_n &\le a_3(n_0,n_1),
\end{aligned}
\end{equation}
for any  $k_0 \geq n_0$. We then have,
\begin{equation}\label{q:ugronest2}
       \x_n \le \Bigl( \frac{a_3(n_0,n_1)}{\dt n_1}
         + a_2(n_0,n_1) \Bigr)\, {\rm e}^{4a_1(n_0,n_1)},
\end{equation}
for any $n \geq n_0 + n_1$.
\end{Lem}
\begin{proof}
Let $n_3$ and $n_4$ be such that
$n_0 \leq  n_2 < n_3 \le n_2 + n_ 1 $.
Using recursively \eqref{q:gronseq2}, we derive
\begin{equation}\begin{aligned}
  \x_{n_2+n_1} &\le \frac{1}{\prod_{n=n_3}^{n_2+n_1} (1-\dt \g_n)}\x_{n_3-1} +
    +\dt \sum_{n=n_3}^{n_2+n_1}\frac{1}{\prod_{j=n}^{n_2+n_1} (1-\dt \g_j)}  \h_n.
\end{aligned}
\end{equation}
Using the fact that $1-x \ge \textrm{e}^{-4x}, \,\forall x \in \left(0,\frac{1}{2}\right)$,
and recalling assumptions ${\eqref{q:groncond2}}_1$ and
${\eqref{q:groncond2}}_2$, we obtain
\[
  \x_{n_2+n_3} \le (\x_{n_3-1} + a_2) e^{-4a_1}.
\]
Multiplying this inequality by $\dt$, summing $n_3$ from $n_2+1$ to $n_2+n_1$
and using assumption ${\eqref{q:groncond2}}_3$ gives the conclusion
\eqref{q:ugronest2} of the lemma.
\end{proof}

We are now able to derive an upper bound on $\|\te^n\|$, $n \geq N$. More precisely, we have the following:

\begin{Lem}\label{infinite}

Let $\{v_0, \te_0\} \in V$ 
and $\{v^n,\te^n\}$ be the solution of the numerical
scheme \eqref{1.25}--\eqref{1.26}. Also, let $T>0$  be arbitrarily fixed and $\dt$ be such that
\begin{equation}\label{t1}
\begin{split}
      \dt  &\leq \min \left\{\kappa_4(\|v_0\|,|\te_0|),
\kappa_5(\|v_0\|,|\te_0|), \frac{T}{2} \right\},
\end{split}
\eeq
where $\kappa_4(\cdot, \cdot)$ is given in Theorem \ref{mainth} and $\kappa_5(\cdot, \cdot)$ is given in Lemma \ref{finite}.
Then there exists $M_2=M_2(\|v_0\|,|\te_0|, T)$, given in \eqref{M2} below, such that
\begin{equation}\label{1.86}
\|\te^n\| \leq M_2(\|v_0\|,|\te_0|, T), \, \forall n \geq N:= \lfloor T/\dt \rfloor.
\eeq
\end{Lem}
\begin{proof}
We  apply Lemma \ref{t:dugronwall2} to \eqref{1.82}, which we rewrite as
\beq\label{1.87}
\begin{split}
   (1-c_2 K_1^2 K_4^2\dt) \|\te^n\|^2 - \|\te^{n-1}\|^2
      & + \|\te^n - \te^{n-1}\|^2 + \kappa {\dt} |A_2 \te^n|^2 \\
      &  \leq \frac{2}{\kappa} K_1^2 \dt.
\end{split}
\eeq
We set $\x_n=\|\te^n\|^2$, $\g_n= c_2 K_1^2 K_4^2$, $\h_n=\frac{2}{\kappa} K_1^2$, $n_0=1$, $n_1=N-1$ and for $k_0 \geq 1$ we compute:
\begin{equation*}
 \dt \sum_{n=k_0}^{k_0+n_1} \g_n = \dt \sum_{n=k_0}^{k_0+n_1} c_2 K_1^2 K_4^2
            \le c_2 K_1^2 K_4^2 T,
\end{equation*}
\begin{equation*}
\begin{aligned}
 \dt \sum_{n=k_0}^{k_0+n_1}  \h_n = \dt \sum_{n=k_0}^{k_0+n_1} \frac{2}{\kappa} K_1^2\le  \frac{2}{\kappa} K_1^2 T,\\
 \dt \sum_{n=k_0}^{k_0+n_1}  \x_n = \dt \sum_{n=k_0}^{k_0+n_1}  \|\te^n\|^2 \le  \frac{1}{\kappa} \left(M_1^2 + \frac{K_1^2}{\kappa} T\right)
(\text{by } \eqref {q:tebdinth}).
\end{aligned}
\end{equation*}
Then Lemma \ref{t:dugronwall2} implies
\begin{equation}\label{M2}
\begin{split}
\|\te^n\|^2 &\leq \frac{2}{\kappa} \left(\frac{M_1^2}{T} + \frac{K_1^2}{\kappa} + K_1^2 T\right)e^{4c_2 K_1^2 K_4^2 T}\\
&:=M_2^2(\|v_0\|, |\te_0|, T), \, \forall n \geq N.
\end{split}
\end{equation}
Thus, the lemma is proved.
\end{proof}
Combining Lemma \ref{finite} and Lemma \ref{infinite}, we obtain that $\te^n$ are uniformly bounded in $V$, for all $n \geq 0$. More precisely, we have

\begin{Thm}\label{t:vte}
Let $\{v_0, \te_0\} \in V$ 
and $\{v^n,\te^n\}$ be the solution of the numerical
scheme \eqref{1.25}--\eqref{1.26}. Also, let $T>0$  be arbitrarily fixed and $\dt$ be such that
\begin{equation}
\begin{split}
      \dt  &\leq \min \left\{\kappa_4(\|v_0\|,|\te_0|),
\kappa_5(\|v_0\|,|\te_0|), \frac{T}{2} \right\}=:\kappa_6(\|v_0\|,|\te_0|),
\end{split}
\eeq
where $\kappa_4(\cdot, \cdot)$ is given in Theorem \ref{mainth} and $\kappa_5(\cdot, \cdot)$ is given in Lemma \ref{finite}.
Then there exists $M_3=M_3(\|v_0\|,\|\te_0\|)$, such that
\begin{equation}\label{1.86}
\|\te^n\| \leq M_3(\|v_0\|,\|\te_0\|), \, \forall n \geq 0.
\eeq
\end{Thm}

\begin{proof}
Taking
\[
M_3(\|v_0\|,\|\te_0\|)=\max\left\{4^{c_2K_1^2K_4^2 T} \left( \|\te_0\|^2+\frac{2}{\kappa K_4^2}  \right), M_2(\|v_0\|,|\te_0|, T)\right \},
\]
Lemmas \ref{finite} and \ref{infinite} give conclusion \eqref{1.86} of the theorem.
\end{proof}

\begin{Cor}\label{4.2}
Under the assumptions of Theorem \ref{t:vte}, we also have
\beq\label{4.63a}
\begin{split}
    \sum_{n=i}^{m} \|\te^n - \te^{n-1}\|^2  \leq &M_3^2+ c_2 K_1^2 K_4^2 M_3^2 \dt (m-n+1)\\
    &+  \frac{2}{\kappa} K_1^2 \dt (m-n+1),  \quad\forall \, i=1, \cdots,m.
\end{split}
\eeq
\end{Cor}

\begin{proof}
Taking the sum of \eqref{1.82} with $n$ from $i$ to $m$ and using \eqref{1.86} gives conclusion \eqref{4.63a} of the corollary right away. 
\end{proof}

Theorem \ref{mainth} and Theorem \ref{t:vte} can be combined to obtain the following

\begin{Thm}\label{thm}
Let $\{v_0, \te_0\} \in V$ 
and $\{v^n,\te^n\}$ be the solution of the numerical scheme
\eqref{1.25}--\eqref{1.26}. 
Then there exists a decreasing positive function $\kappa_7(\cdot)$
and an increasing  positive function $K_5(\cdot)$, such that if
\begin{equation}
\begin{split}
      \dt  &\leq \kappa_7(\|\{v_0, \te_0\}\|),
\end{split}
\eeq 
then
\begin{equation}\label{1.87}
\|\{v^n, \te^n\}\| \leq K_5(\|\{v_0, \te_0\}\|), \, \forall n \geq 0.
\eeq
\end{Thm}
\section{Convergence of Attractors}\label{s5}

In this section we address the issue of the convergence of the
attractors generated by the discrete system
\eqref{2.45a}--\eqref{1.26} to the attractor generated by the
continuous system \eqref{1.11}--\eqref{1.14c}. Whereas for the
continuous system \eqref{1.11}--\eqref{1.14c} one can prove both the
existence and uniqueness of the solution (see, e.g.,
\cite{temam:iddsmp})--and, therefore, define a global attractor--,
for the discrete system  \eqref{2.45a}--\eqref{1.26} one can prove
(using Theorem \ref{thm}) the uniqueness of the solution provided
that $k\leq \kappa(\|u_0\|)$, for some $ \kappa(\|u_0\|)>0$. Since
the time restriction depends on the initial data, one cannot define
a single-valued attractor in the classical sense, and this is why we
need to use the attractor theory for the so-called multi-valued
mappings. Multi-valued dynamical systems have been investigated by
many authors (see, e.g., \cite{Ball}, \cite{Barbu}, \cite{CLM},
\cite{MV}, \cite{RSS}, \cite{SSS}), but in this article we use the
tools developed in \cite{CZT} to study the convergence of the
discrete (multi-valued) attractors to the continuous (single-valued)
attractor.  For convenience, we recall those results in Subsection
\ref{ssec:abs}, and  then we apply them to the thermohydraulics
equations in Subsection \ref{ss:5.2}.



\subsection{Attractors for multi-valued mappings}\label{ssec:abs}
Throughout this subsection, we consider $(H, |\cdot|)$ to be a Hilbert space and $\T$ to be either $\Real^+=[0,\infty)$ or $\N$.
\begin{Def}
A one-parameter family of set-valued maps $S(t):2^H\to 2^H$
is a \textbf{multi-valued semigroup} (m-semigroup) if it satisfies the following properties:

\begin{enumerate}
    \item[(S.1)] $S(0)=I_{2^H}$ (identity in $2^H$); \label{S.1}
    \item[(S.2)] $S(t+s)=S(t)S(s)$, for all $t,s \in \T$. \label{S.2}
\end{enumerate}

Moreover, the m-semigroup is said to be \textbf{closed} if 
      $S(t)$ is a closed map for every $t\in\T$, meaning that if $x_n\to x$ in $H$ and
    $y_n\in S(t)x_n$ is such that $y_n\to y$ in $H$, then $y\in S(t)x$. (To simplify the notation, hereafter we have written $S(t)x$ in place of $S(t)\{x\}$.)
\end{Def}

\begin{Def}
The \textbf{positive orbit} of $\B$, starting at $t\in\T$, is the set
$$
\gamma_t(\B)=\bigcup_{\tau\geq t}S(\tau)\B,
$$
where
$$
S(t)\B=\bigcup_{x\in\B}S(t)x.
$$
\end{Def}

\begin{Def}
For any $\B\in2^H$, the set
$$
\omega(\B)=\bigcap_{t\in\T}\overline{\gamma_t(\B)}
$$
is called the \textbf{$\omega$-limit set} of $\B$.
\end{Def}

\begin{Def}
A nonempty set $\B\in 2^H$ is \textbf{invariant} for $S(t)$ if
$$
S(t)\B=\B, \qquad \forall t\in\T.
$$
\end{Def}
\begin{Def}
A set $\B_0\in 2^H$ is an \textbf{absorbing set} for the m-semigroup $S(t)$
if for every bounded set $\B\in 2^H$ there exists $t_\B\in \T$ such that
$$
S(t)\B\subset \B_0, \qquad \forall t\geq t_\B.
$$
\end{Def}

\begin{Def}
A nonempty set $\C\in 2^H$ is \textbf{attracting} if for every bounded set $\B$ we have
$$
\lim_{t\to \infty}\dist(S(t)\B,\C)=0,
$$
where $\dist(\cdot,\cdot)$ is
the \textbf{ Hausdorff semidistance}, defined as
\beq\label{5.1a}
\dist(\B,\C)=\sup_{b\in\B}\inf_{c\in\C}|b-c|, \forall \, \B, \C \subset  H.
\eeq
\end{Def}


\begin{Def}
A nonempty compact set $\A\in 2^X$ is said to be the \textbf{global attractor} of $S(t)$ if $\A$ is an invariant attracting set.
\end{Def}

\begin{Rem}
The global attractor, if it exists, is necessarily unique. Moreover, it enjoys the following maximality
and minimality properties:
\begin{enumerate}
    \item[(i)] if $\tilde\A$ is a bounded invariant set, then $\A\supset \tilde\A$;
    \item[(ii)] if $\tilde\A$ is a closed attracting set, then $\A\subset \tilde\A$.
\end{enumerate}
\end{Rem}

\begin{Def}
Given a bounded set $\B\in 2^H$, the \textbf{Kuratowski measure of noncompactness} $\alpha(\B)$ of
$\B$ is defined as
$$
\alpha(\B)=\inf\big\{\delta\, :\, \B\ \text{has a finite cover by balls of } X \text{ of diameter less than } \delta\big\}.
$$
\end{Def}

We note that $\alpha(\B)=0$ if and only if $\overline{\B}$ is compact.

\vskip .5cm

The following theorem, whose proof can be found in \cite{CZT}, gives conditions under which a global attractor exists.
\begin{Thm}\label{teo:attr}
Suppose that the closed m-semigroup $S(t)$ possesses a bounded absorbing set $\B_0\in 2^H$ and
\begin{equation}\label{eq:asymptcpt}
\lim_{t\to\infty}\alpha(S(t)\B_0)=0.
\end{equation}
Then $\omega(\B_0)$ is the global attractor of $S(t)$.
\end{Thm}

 For the purpose of this article, we need to introduce the notion of \textit{discrete m-semigroups}. More precisely, we have the following:
 \begin{Def}
 Given a set-valued map $S:2^H\to 2^H$, we define a \textbf{discrete m-semigroup} 
 by
$$
S(n)=S^n, \qquad \forall n\in\N,
$$
and we will denote it by $\{S\}_{n\in\N}$
(instead of $\{S^n\}_{n\in\N}$).
\end{Def}

\begin{Rem}
Given two nonempty sets $\B,\C\in 2^H$, we write
$$
\B-\C=\{b-c:\, b\in\B,\, c\in\C\} \qquad \text{and} \qquad |\B|=\sup_{b\in\B} |b|.
$$
\end{Rem}

In order to prove the convergence of the attractors generated by
the discrete system \eqref{2.45a}--\eqref{1.26} to the attractor generated by the continuous
system \eqref{1.11}--\eqref{1.14c} we will use the following result, whose proof can be found in \cite{CZT}; see also \cite{SW}, \cite{TW2011}.
\begin{Thm}\label{teo:approx}
Let $S(t)$ be a closed m-semigroup, possessing the global attractor $\A$, and for $\kappa_0>0$,
let $\{S_k,\, 0<k\leq \kappa_0\}_{n\in\N}$ be a family of discrete closed m-semigroups,
with global attractor $\A_k$. Assume the following:
\begin{enumerate}
    \item[(H1)]  [Uniform boundedness]:  there exists $\kappa_1\in (0,\kappa_0]$ such that the set
    $$
    \K=\bigcup_{k\in (0,\kappa_1]}\A_k
    $$
    is bounded in $H$;  \label{H1}
    \item[(H2)] [Finite time uniform convergence]: there exists $t_0\geq 0$ such that for any $T^\star>t_0$,
    $$
    \lim_{k\to 0}\sup_{x\in\A_k,\,nk\in [t_0,T^\star]}| S_k^nx-S(nk)x|=0.
    $$
    \label{H2}
\end{enumerate}
Then
$$
\lim_{k\to 0}\dist (\A_k,\A)=0,
$$
where $\dist$ denotes the Hausdorff semidistance defined in \eqref{5.1a}.
\end{Thm}

\subsection{Application: The thermohydraulics equations}\label{ss:5.2}
The system \eqref{1.11}--\eqref{1.14c} possesses a unique solution
and thus generates a continuous single-valued dynamical system
$S(t):H\to H$, with  global attractor $\A$, bounded in $V$ (see,
e.g., \cite{temam:iddsmp}). Using Theorem \ref{thm} one can prove
that the discrete system \eqref{2.45a}--\eqref{1.26} has a unique
solution provided that $k\leq \kappa(\|u_0\|)$, for some $
\kappa(\|u_0\|)>0$. The dependence of the time step $k$ on the
initial data prevents us from defining a single-valued attractor in
the classical sense, but this difficulty can be overcome by the
theory of the multi-valued attractors. More precisely, in this
article we will prove that there exists $\kappa_0>0$ such that if
$0<k\leq \kappa_0$,  the system \eqref{2.45a}--\eqref{1.26}
generates a closed discrete m-semigroup $\{S_k\}_{n\in\N}$, with
global attractors $\A_k$, that will converge to $\A$ in the sense of
Theorem \ref{teo:approx}.

In order to do that,  we define, for $\dt>0$, the multi-valued map $S_k:2^H\to 2^H$ as follows: for every $\tilde{u}=\{\tilde{v}, \tilde{\te}\} \in H $,
$$
S_k \tilde{u}=\{u=\{v,\te\} \in V:\, u \text{ solves \eqref{5.1}--\eqref{5.2} below with time-step } k\}:
$$
\begin{gather}
(v, v')+ \nu \dt((v, v'))+\dt b_1(v, v, v')- \dt(e_2 \theta, v') = (\tilde{v}, v'), \, \forall v'\in V_1,\label{5.1}\\
   (\theta, \te') + \kappa \dt((\theta, \te')) +\dt b_2(v ,\theta, \te') - \dt (v_2, \te') = (\tilde{\te}, \te'), \, \forall \te' \in V_2.\label{5.2}
\end{gather}

We then have the following:
\begin{Thm}
The multi-valued map $S_k$ associated with the implicit Euler scheme \eqref{2.45a}--\eqref{1.26}  generates
a closed discrete m-semigroup $\{S_k\}_{n\in\N}$.
\end{Thm}
\begin{proof}
Since conditions (S.1) and (S.2) are satisfied by definition, we just need to prove
that for each $n\in \N$, $S_k^n$ is a closed multi-valued map. For that, we let $n\in \N$ be arbitrarily fixed and,  as $j\to \infty$, we let $u^0_j\to u^0$ in $H$, where
$u_j^0=\{v_j^0,\te_j^0\}, u^0=\{v^0,\te^0\}$. Also let $u^n_j\in S^n_k u^0_j$ be such that $u^n_j\to u^n$ in $H$, where $u_j^n=\{v_j^n,\te_j^n\}, u^n=\{v^n,\te^n\}$. We need to show that $u^n\in S_k^nu^0$.

Indeed, since $u^n_j\in S^n_k u^0_j$, there exists a sequence
$(u_j^0, u_j^1,\ldots, u_j^{n-1},u_j^n)$, with  $u_j^i\in S_k u_j^{i-1}$, such that
\begin{gather}
(v_j^i, v')+ \nu \dt((v_j^i, v'))+\dt b_1(v_j^i, v_j^i, v')- \dt(e_2 \theta_j^i, v') = (v_j^{i-1}, v'), \, \forall v'\in V_1,\label{5.4}\\
   (\theta_j^i, \te') + \kappa \dt((\theta_j^i, \te')) +\dt b_2(v_j^i,\theta_j^i, \te') - \dt ((v_j^i)_2, \te') = ({\te}_j^{i-1}, \te'), \, \forall \te' \in V_2.\label{5.5}
\end{gather}
The sequence $u^0_j$ being convergent in $H$, it is also bounded in $H$ and thus there exists $M>0$ such that
\beq
\sup_j |u_j^0|^2\leq M.
\eeq
Then Lemmas \ref{max} and \ref{t:bdh} imply that for every $i=1,\ldots, n$, the sequences $v^i_j$ and $\te^i_j$ are bounded in $V_1$ and $V_2$, respectively. We therefore have that there exist subsequences still denoted $v^i_j$ and $\te^i_j$, such that as $j\to\infty$:
\begin{gather}
v^i_j\to v^i, \text{ strongly in } H_1 \text{ and weakly in } V_1,\\
\te^i_j\to \te^i, \text{ strongly in } H_2 \text{ and weakly in } V_2.
\end{gather}
Now, passing to the limit in \eqref{5.4}--\eqref{5.5}, we obtain
\begin{gather}
(v^i, v')+ \nu \dt((v^i, v'))+\dt b_1(v^i, v^i, v')- \dt(e_2 \theta^i, v') = (v^{i-1}, v'), \, \forall v'\in V_1,\label{5.9}\\
   (\theta^i, \te') + \kappa \dt((\theta^i, \te')) +\dt b_2(v^i,\theta^i, \te') - \dt ((v^i)_2, \te') = ({\te}^{i-1}, \te'), \, \forall \te' \in V_2.\label{5.10}
\end{gather}
We therefore obtain that $u^i\in S_k u^{i-1}$, 
for each $i=1,\ldots, n$, and hence, $u^n\in S_k u^{n-1}\subset S_k^n u^0$.
This completes the proof of the theorem.
\end{proof}


In order to prove the existence of the discrete global attractors,
we first prove the existence of absorbing sets. More precisely, we
have the following:

\begin{Prop}\label{prop:Vabs}
There exists $\kappa_8>0$, independent of $\{v_0,\te_0\},n,k$, such that
if $k\in(0,\kappa_8]$ the following holds: there exists a constant
$R_1>0$ such that for every $R\geq 0$ and $|\{v_0,\te_0\}|\leq R$,
there exists $N_1=N_1(R, \dt)\geq 0$ such that
\begin{equation}\label{eq:BBBB1}
\|S_k^n \{v_0,\te_0\}\|\leq R_1, \qquad \forall n\geq N_1.
\end{equation}
Hence, the set
$$
\B_1=\{\{v,\te\}\in V:\, \|\{v,\te\}\|\leq R_1\}
$$
is a $V$-bounded absorbing set for $\{S_k\}_{n\in\N}$, for $k\in(0,\kappa_8]$.
\end{Prop}
\begin{proof}
Let $\kappa_1$ be as in Corollary 3.2 and let $k \leq
\min\{1,\kappa_1\}$. Also, let $R\geq 0$ and $|\{v_0, \te_0\}|\leq
R$. Then, by Corollary 3.2, there exists $N_0=N_0(R, \dt)\geq 0$
such that
\begin{equation}\label{R_0}
    |\{v^n, \te^n\}|\leq \rho_0, \qquad \forall n\geq N_0.
    \end{equation}
Let $m := N_0+\Big\lfloor\frac{1}{k}\Big\rfloor$. Then equations
\eqref{q:vbdinth} and \eqref{q:tebdinth} imply

\beq\label{5.13}
\nu {\dt} \sum_{j=N_0+1}^{m} \|{v}^j\|^2
      \leq \rho_0^2 +\frac{1}{\nu} \rho_0^2 (m-N_0)\dt,
 \eeq
\beq\label{5.14}
  \kappa {\dt} \sum_{j=N_0+1}^{m} \|\te^j\|^2\leq \rho_0^2 +
             \frac{1}{\kappa}\rho_0^2 (m-N_0)\dt.
\end{equation}
Adding the above relations we obtain
 \beq\label{5.15}
{\dt}\left(\sum_{j=N_0+1}^{m} (\nu \|{v}^j\|^2 + \kappa \|\te^j\|^2)
\right) \leq \rho_0^2\left(2 +\frac{1}{\nu}(m-N_0)\dt +
\frac{1}{\kappa}(m-N_0)\dt\right).
 \eeq
Assuming that for every $j\in\{N_0+1, \cdots, m\}$
 \begin{equation*}\label{5.16}
  \begin{split}
 (\nu \|{v}^j\|^2 + \kappa \|\te^j\|^2) \geq \frac{\rho_0^2}{{\dt} (m-N_0)}\left(2 +\frac{1}{\nu}(m-N_0)\dt +
\frac{1}{\kappa}(m-N_0)\dt\right), \\
\end{split}
\end{equation*}
 we obtain
 \beq\label{5.16'}
{\dt}\left(\sum_{j=N_0+1}^{m} (\nu \|{v}^j\|^2 + \kappa \|\te^j\|^2)
\right) \geq \rho_0^2\left(2 +\frac{1}{\nu}(m-N_0)\dt +
\frac{1}{\kappa}(m-N_0)\dt\right),
 \eeq
 which contradicts \eqref{5.15}.
 Hence there exists $l\in\{N_0+1, \cdots, m\}$ such that
 \beq\label{5.16}
 \begin{split}
 (\nu \|{v}^l\|^2 + \kappa \|\te^l\|^2) &\leq \frac{\rho_0^2}{{\dt} (m-N_0)}\left(2 +\frac{1}{\nu}(m-N_0)\dt +
\frac{1}{\kappa}(m-N_0)\dt\right)\\
&\leq 2{\rho_0^2}\left(2 +\frac{1}{\nu} + \frac{1}{\kappa}\right).
\end{split}
 \eeq
We, therefore, have
 \beq\label{5.17}
\|\{{v}^l, \te^l\}\|^2 \leq 2{\rho_0^2}\left(2 +\frac{1}{\nu} +
\frac{1}{\kappa}\right) \left(\frac{1}{\nu} +
\frac{1}{\kappa}\right)=:R_*^2.
  \eeq
  Applying Theorem \ref{thm} with initial data $\{v^{l},\te^l\}$ we obtain that there exists
  $\kappa_7(\|\{v^{l},\te^l\}\|)$ and $K_5(\|\{v^{l},\te^l\}\|)$
  such that if $k \leq \kappa_7(\|\{v^{l},\te^l\}\|)$, then
\begin{equation}\label{5.18}
\|\{v^n, \te^n\}\|\leq K_5(\|\{v^{l},\te^l\}\|), \forall n \geq {l}.
\end{equation}
Recalling \eqref{5.17} and the fact that $\kappa_7(\cdot)$ and
$K_5(\cdot)$ are, respectively, decreasing and increasing functions of their
arguments, \eqref{5.18} yields
  \beq\label{5.19}
  \|\{v^n, \te^n\}\|\leq K_5(R_*)=:R_1, \forall n \geq N_1=N_1(R, \dt):=N_0+\Big\lfloor\frac{1}{k}\Big\rfloor,
   \eeq
provided that $\dt\leq \kappa_8$, where
\beq\label{5.19}
\kappa_8=\min\{1,\kappa_1, \kappa_7(R_*)\}.
 \eeq
 This completes the proof of
Proposition \ref{prop:Vabs}.
\end{proof}

We are now in a position to prove the existence of the discrete
global attractors. More precisely, we have the following:
%
\begin{Prop}
For every $k\in (0,\kappa_8]$, there exists the global attractor
$\A_k$ of the m-semigroup $\{S_k\}_{n\in \N}$.
\end{Prop}

\begin{proof}
Let $\B_0=B_{H}(0, \rho_0)$ be the bounded absorbing set given in
Corollary \ref{C2}. Then Proposition \ref{prop:Vabs} implies that
$S_k^n \B_0$ is bounded in $V$, for all $n\geq N_1(\rho_0, \dt)$.
Since $V$ is compactly embedded in $H$, we obtain that $S_k^n \B_0$
is relatively compact in $H$ and, thus, $\alpha(S_k^n\B_0)=0$, for
all $n\geq N_1(\rho_0, \dt)$. Condition \eqref{eq:asymptcpt} of
Theorem \ref{teo:attr} is therefore satisfied and then the existence
of the discrete global attractor $\A_k$ follows right away.
\end{proof}


\begin{Rem}
Since the global attractor $\A_k$ is the smallest closed attracting set of $H$, Proposition \ref{prop:Vabs} implies
\begin{equation}\label{5.21}
\A_k\subset \B_1, \forall k\in(0,\kappa_8],
\end{equation}
and thus
\begin{equation}\label{eq:unifbdd}
\bigcup_{k\in(0,\kappa_8]}\A_k \subset \B_1.
\end{equation}
\end{Rem}

Let us recall that our goal is to prove, using Theorem \ref{teo:approx}, that the discrete global attractors $\A_k$ converge to the continuous global attractor $\A$. 
Thanks to \eqref{eq:unifbdd}, condition (H1) of Theorem \ref{teo:approx} holds true. There remains to prove the finite time uniform convergence required by (H2). In order to do that, we 
define, for any $\dt>0$ and for any function $\psi$, the following:
\begin{equation}
\psi_k(t)=\psi^n, \quad t \in [(n-1)k, nk),
\end{equation}
\begin{equation}
\tilde{\psi}_k(t)= \psi^n+\frac{t-nk}{k}(\psi^n-\psi^{n-1}), \quad
 t\in [(n-1)k, nk).
\end{equation}

With the above notations, equations \eqref{1.25} and \eqref{1.26}
can be rewritten as follows; for $t \in [(n-1)k, nk)$:
\begin{gather}
 \left(\frac{ \partial \tilde{{v}}_k(t)}{\partial t}, v\right) + \nu ((\tilde{v}_k(t),v))+b_1(\tilde{{v}}_k(t),\tilde{v}_k(t),v)  = (e_2 \tilde{\theta}_k(t), v)+({f}_k(t), v),\, \forall v\in V_1,\label{1.111}\\
 \left(  \frac{ \partial \tilde{\theta}_k(t)}{\partial t}, \te\right) + \kappa ((\tilde{\theta}_k(t), \te))+ b_2(\tilde{v}_k(t), \tilde{ \theta}_k(t),\te) - (\tilde{v}_k(t))_2, \te)  = (g_k(t),\te), \, \forall \te \in V_2, \label{1.121}
\end{gather}
where
\begin{equation}
\begin{split}
({f}_k(t), v)= & \, \, \nu ((\tilde{v}_k(t)-v_k(t),v))+b_1(\tilde{{v}}_k(t),\tilde{v}_k(t),v)\\& -b_1({{v}}_k(t),{v}_k(t),v) 
 -(e_2 (\tilde{\theta}_k(t)-\theta_k(t)) ,v),
\end{split}
\end{equation}
\begin{equation}
\begin{split}
(g_k(t),\te)=& \,\, \kappa ((\tilde{\theta}_k(t)-{\theta}_k(t), \te))+b_2(\tilde{v}_k(t), \tilde{ \theta}_k(t),\te)\\
&-b_2({v}_k(t), { \theta}_k(t),\te)-((\tilde{v}_k(t)-v_k(t))_2,\te).
%
\end{split}
\end{equation}

\begin{Lem}\label{fk, gk}
Let $T^*>0$ be arbitrarily fixed and let $\dt\leq \kappa_0$, where
\begin{equation}\label{kappa0}
\kappa_0=\min\{\kappa_8, \kappa_7(R_1)\}, 
\end{equation}
with $\kappa_8$ being given in \eqref{5.19} and $\kappa_7$ being given in Theorem \ref{thm}. 
Assume that $\{{v}_0,\te_0\} \in \mathcal{A}_k$ and let $\{{v}^n, \te^n\}$ be the
solution of the numerical scheme \eqref{2.45a}--\eqref{1.26}. Then
there exist $K_{6}( T^*)$ and $K_{7}(T^*)$ such that
\begin{equation}\label{fk}
\|{f}_k\|_{L^2(0,T^*;V_1')}^2\leq \dt K_{6}(T^*),
\end{equation}
and
\begin{equation}\label{gk}
\|g_k\|_{L^2(0,T^*;V_2')}^2\leq \dt K_{7}(T^*).
\end{equation}
\end{Lem}
\begin{proof}
Let us first note that for any $t \in [(n-1)k, nk)$ we have
\begin{equation}\label{4.12}
\begin{split}
\tilde{\psi}_k(t)-{\psi}_k(t)=\frac{t-nk}{k}(\psi^n-\psi^{n-1}).
\end{split}
\end{equation}

Also, since $\{v_0,\te_0\} \in \mathcal{A}_k$, we have that $\|\{v_0,\te_0\}\|\leq R_1$ (by \eqref{5.21}) and then Theorem \ref{thm} implies that for $k \leq \kappa_0$, 
\begin{equation}\label{5.33}
\|\{v^n, \te^n\}\| \leq K_5(R_1), \, \forall n \geq 0.
\end{equation}

Now let $v\in V_1$ be such that $\|v\|\leq 1$, and let $t\in [(n-1)k,nk)$ be fixed. 
Using property \eqref{b1.1} of the trilinear form $b_1$, we have
\begin{equation}\label{4.13}
\begin{split}
&|b_1(\tilde{{v}}_k(t),\tilde{v}_k(t),v) -b_1({{v}}_k(t),{v}_k(t),v)|\\
&=|b_1(\tilde{v}_k(t)-{v}_k(t), \tilde{v}_k(t),v)+b_1({v}_k(t), \tilde{v}_k(t)-{v}_k(t),v)|\\
&\leq c_b(\|\tilde{v}_k(t) -v_k(t)\|
(\|\tilde{v}_k(t)\|+\|v_k(t)\|)\|v\|\\
& \leq c \|v^n-v^{n-1}\| \quad (\text{by }
\eqref{4.12}, \eqref{5.33} \text{ and } \|v\|\leq 1).
\end{split}
\end{equation}
We also have
\begin{equation}\label{4.14}
\nu |((\tilde{v}_k(t)-v_k(t),v))|
\leq
\nu \|v^n-v^{n-1}\|,
\end{equation}
\begin{equation}\label{4.15}
|( e_2 (\tilde{\theta}_k(t)-\theta_k(t)), v)|\leq
\|\te^n-\te^{n-1}\|.
\end{equation}
Relations \eqref{4.13}--\eqref{4.15} imply
\begin{equation}\label{4.16}
\|{f}_k(t)\|_{V_1'}\leq
c(\|v^n-v^{n-1}\|+\|\te^n-\te^{n-1}\|),
\end{equation}
and thus, setting $N^*=\lfloor T^\star/k \rfloor$ and recalling that $\|\{v_0,\te_0\}\|\leq R_1$ , we obtain
\begin{equation}\label{4.17}
\begin{split}
\|{f}_k\|_{L^2(0,T^*;V_1')}^2=&\int_0^{T^*
}\|{f}_k(t)\|_{V_1'}^2 dt=\sum_{n=1}^{N^*+1} \int_{(n-1)\dt}
^{n\dt}\|{f}_k(t)\|_{V_1'}^2 dt\\
& \leq \dt K_{6}( T^*) \quad (\text{by } \eqref{4.16}, \eqref{4.41a}, \eqref{4.63a}),
\end{split}
\end{equation}
which proves \eqref{fk}.

Now let $\te\in V_2$ be such that $\|\te\|\leq 1$, and let $t\in [(n-1)k,nk)$ be fixed. Using property \eqref{b2} of the trilinear form $b_2$, we have
\begin{equation}\label{4.18}
\begin{split}
&|b_2(\tilde{v}_k(t), \tilde{ \theta}_k(t),\te)
-b_2({v}_k(t), { \theta}_k(t),\te)|\\
&=|b_2(\tilde{v}_k(t)-{v}_k(t), \tilde{\te}_k(t),\te)+b_2({v}_k(t), \tilde{\te}_k(t)-{\te}_k(t),\te)|\\
&\leq c_b(\|\tilde{v}_k(t)-v_k(t)\| \|\tilde{\te}_k(t)\|+\|v_k(t)\|\|\tilde{\te}_k(t) -\te_k(t)\|) \|\te\|\\
& \leq c (\|v^n-v^{n-1}\|+ \|\te^n-\te^{n-1}\|) \quad (\text{by }
\eqref{4.12}, \eqref{5.33} \text{ and } \|\te\| \leq 1).
\end{split}
\end{equation}
We also have
\begin{equation}\label{5.40}
\kappa |((\tilde{\theta}_k(t)-{\theta}_k(t), \te))|
\leq
\kappa \|\te^n-\te^{n-1}\|,
\end{equation}
\begin{equation}\label{5.41}
|((\tilde{v}_k(t)-v_k(t))_2,\te)|\leq
\|v^n-v^{n-1}\|.
\end{equation}

Relations \eqref{4.18}--\eqref{5.41} imply
\begin{equation}\label{5.42}
\|g_k(t)\|_{V_2'}\leq
c(\|v^n-v^{n-1}\|+\|\te^n-\te^{n-1}\|),
\end{equation}
and thus setting $N^*=\lfloor T^\star/k \rfloor$ and recalling that $\|\{v_0,\te_0\}\|\leq R_1$ , we obtain
\begin{equation}\label{5.43}
\begin{split}
\|g_k\|_{L^2(0,T^*;V_2')}^2=&\int_0^{T^* }\|g_k(t)\|_{V_2'}^2
dt=\sum_{n=1}^{N^*+1} \int_{(n-1)\dt}
^{n\dt}\|g_k(t)\|_{V_2'}^2 dt\\
& \leq \dt K_{7}( T^*) \quad (\text{by } \eqref{5.42}, \eqref{4.41a}, \eqref{4.63a}),
\end{split}
\end{equation}
which proves \eqref{gk} and the proof of the lemma is complete.
\end{proof}
We are now able to prove that condition (H2) of Theorem \ref{teo:approx} is satisfied. More precisely, we have the following

\begin{Prop}[Finite time uniform convergence]
For any $T^*>0$ we have
\begin{equation}
\lim_{k \to 0} \sup_{\{v_0,\te_0\} \in \mathcal{A}_k, \,nk
\in [0,T^*]} |S_k^n
\{v_0,\te_0\}-S(nk)\{v_0,\te_0\}|=0.
\end{equation}
\end{Prop}

\begin{proof}
Let
\begin{equation}
{{\xi}}_k(t)=v(t)-\tilde{v}_k(t), \qquad
\eta_k(t)=\te(t)-\tilde{\te}_k(t).
\end{equation}
Subtracting \eqref{1.111} and \eqref{1.121} from \eqref{1.11} and
\eqref{1.12} written in their week form, respectively, we obtain
\begin{equation}\label{5.46}
\begin{split}
  &\left( \frac{ \partial {\xi}_k(t)}{\partial t},v'\right)   + \nu ((\xi_k(t),v'))+b_1(\xi_k(t) ,{v}(t),v')\\&+b_1(\tilde{v}_k(t),\xi_k(t),v')
  = (e_2 {\eta}_k(t),v')-(f_k(t),v'),\, \forall v'\in V_1,
   \end{split}
   \end{equation}
\begin{equation}\label{5.47}
\begin{split}
&\left( \frac{ \partial{\eta}_k(t)}{\partial t},\te' \right) + \kappa  (( {\eta}_k(t), \te')) +b_2(\xi_k(t), { \theta}(t),\te')\\
&+b_2(\tilde{v}_k(t) , {
\eta}_k(t),\te') -
((\xi_k(t))_2,\te') = -(g_k(t),\te'), \, \forall \te' \in V_2.
\end{split}
\end{equation}

Replacing $v'$ by  $\xi_k(t)$ in \eqref{5.46}, we find
\begin{equation}\label{5.48}
\begin{split}
\frac{1}{2} &\frac{d}{dt} |\xi_k(t)|^2+\nu
\|\xi_k(t)\|^2+b_1(\xi_k(t),
v(t), \xi_k(t))\\
&=(e_2 {\eta}_k(t), \xi_k(t))-(f_k(t),
\xi_k(t)).
\end{split}
\end{equation}
Using property \eqref{b1.1} of the form $b_1$, we bound the
nonlinear term as
\begin{equation}\label{5.49}
\begin{split}
b_1(\xi_k(t), v(t), \xi_k(t))&\leq
c_b|\xi_k(t)|
\|\xi_k(t)\|\|v(t)\|\\
&\leq
\frac{\nu}{6}\|\xi_k(t)\|^2+\frac{c}{\nu}|\xi_k(t)|^2\|v(t)\|^2.
\end{split}
\end{equation}
Using the Cauchy--Schwarz inequality, we also have
\begin{equation}\label{5.50}
\begin{split}
|(e_2 {\eta}_k(t), \xi_k(t))|&\leq | {\eta}_k(t)| |\xi_k(t)|\\
&\leq | {\eta}_k(t)| \|\xi_k(t)\|\\
&\leq \frac{\nu}{6}\|\xi_k(t)\|^2+\frac{c}{\nu}|\eta_k(t)|^2,
\end{split}
\end{equation}
\begin{equation}\label{5.51}
\begin{split}
|(f_k(t), \xi_k(t))|&\leq \| f_k(t)\|_{V_1'} \|\xi_k(t)\|\\
&\leq
\frac{\nu}{6}\|\xi_k(t)\|^2+\frac{c}{\nu}\|f_k(t)\|_{V_1'}^2.
\end{split}
\end{equation}
Relations \eqref{5.48}--\eqref{5.51} imply
\begin{equation}\label{5.52}
\begin{split}
\frac{d}{dt} |\xi_k(t)|^2+\nu\|\xi_k(t)\|^2\leq
&\frac{c}{\nu}\|v(t)\|^2|\xi_k(t)|^2\\
&+\frac{c}{\nu}|\eta_k(t)|^2+\frac{c}{\nu}\|f_k(t)\|_{V_1'}^2.
\end{split}
\end{equation}
Now replacing $\te'$ by $\eta_k(t)$  in \eqref{5.47}, we find
\begin{equation}\label{5.53}
\begin{split}
\frac{1}{2} &\frac{d}{dt} |\eta_k(t)|^2+ \kappa \|{\eta}_k(t)\|^2 +b_2(\xi_k(t), \theta(t),\eta_k(t))\\
&-((\xi_k(t))_2, \eta_k(t)) = -(g_k(t), \eta_k(t)).
\end{split}
\end{equation}
Using property \eqref{b2} of the form $b_2$, we bound the nonlinear
term as
\begin{equation}\label{5.54}
\begin{split}
|b_2(\xi_k(t), \theta(t),\eta_k(t))|&\leq
c_b|\xi_k(t)|^{1/2}
\|\xi_k(t)\|^{1/2}\|\theta(t)\||\eta_k(t)|^{1/2}
\|\eta_k(t)\|^{1/2}\\
&\leq
\frac{\nu}{6}\|\xi_k(t)\|^2+\frac{\kappa}{6}\|\eta_k(t)\|^2\\
&+\frac{c}{\nu}\|\theta(t)\|^2|\xi_k(t)|^2+\frac{c}{\kappa}\|\theta(t)\|^2|\eta_k(t)|^2.
\end{split}
\end{equation}
Using the Cauchy--Schwarz inequality, we also have the following
bounds:
\begin{equation}\label{5.55}
\begin{split}
|((\xi_k(t))_2, \eta_k(t))| &\leq |\xi_k(t)| |\eta_k(t)| \\
& \leq \frac{\kappa}{6}\|\eta_k(t)\|^2 +
\frac{c}{\kappa}|\xi_k(t)|^2,
\end{split}
\end{equation}
\begin{equation}\label{5.56}
\begin{split}
|(g_k(t), \eta_k(t))| &\leq \| g_k(t)\|_{V_2'} \|\eta_k(t)\|\\
&\leq
\frac{\kappa}{6}\|\eta_k(t)\|^2+\frac{c}{\kappa}\|g_k(t)\|_{V_2'}^2.
\end{split}
\end{equation}
Relations \eqref{5.53}--\eqref{5.56} imply
\begin{equation}\label{5.57}
\begin{split}
\frac{d}{dt} |\eta_k(t)|^2+\kappa\|\eta_k(t)\|^2\leq
&\frac{\nu}{3}\|\xi_k(t)\|^2+\frac{c}{\nu}\|\theta(t)\|^2|\xi_k(t)|^2\\
&+\frac{c}{\kappa}\|\theta(t)\|^2|\eta_k(t)|^2+\frac{c}{\kappa}|\xi_k(t)|^2\\
&+\frac{c}{\kappa}\|g_k(t)\|_{V_2'}^2.
\end{split}
\end{equation}
Adding equations \eqref{5.52} and \eqref{5.57}, we obtain
\begin{equation}\label{5.58}
\begin{split}
\frac{d}{dt}&(|\xi_k(t)|^2+
|\eta_k(t)|^2)+\frac{2}{3}\nu\|\xi(t)\|^2+\kappa\|\eta(t)\|^2
\\&\leq
\frac{c}{\nu}\left(\|v(t)\|^2+\|\theta(t)\|^2 +\frac{\nu}{\kappa}\right)|\xi_k(t)|^2\\
&+c\left(\frac{1}{\nu}+\frac{1}{\kappa}\|\theta(t)\|^2\right)|\eta_k(t)|^2\\
&+\frac{c}{\nu}\|f_k(t)\|_{V_1'}^2+\frac{c}{\kappa}\|g_k(t)\|_{V_2'}^2.
\end{split}
\end{equation}
As shown in \cite{temam:iddsmp}, the solution $\{v, \theta\}$ of the continuous problem
is uniformly bounded in $V$ for all $t\ge0$. More precisely, we have
\begin{equation}\label{eq:uniS}
\sup_{t\geq 0}\sup_{\{v_0, \te_0\}\in \B_1} \|S(t)\{v_0, \te_0\}\|\leq c.
\end{equation}
Thus, inequality \eqref{5.58} becomes
\begin{equation}\label{5.60}
\begin{split}
\frac{d}{dt}&(|\xi_k(t)|^2+
|\eta_k(t)|^2)+\frac{2}{3}\nu\|\xi(t)\|^2+\kappa\|\eta(t)\|^2
\\&\leq
c(|\xi_k(t)|^2+|\eta_k(t)|^2)+\frac{c}{\nu}\|f_k(t)\|_{V_1'}^2+\frac{c}{\kappa}\|g_k(t)\|_{V_2'}^2.
\end{split}
\end{equation}
By Gronwall's lemma and using the fact that $\xi_k(0)=\eta(0)=0$, we obtain
\begin{equation}\label{5.61}
\begin{split}
|\xi_k(t)|^2+ |\eta_k(t)|^2 &\leq
ce^{cT^*}(\|f_k\|_{L^2(0,T^*;V_1')}^2+\|g_k\|_{L^2(0,T^*;V_2')}^2),
\end{split}
\end{equation}
and recalling \eqref{fk} and \eqref{gk}, we find
\begin{equation}\label{5.62}
\begin{split}
|\xi_k(t)|^2+ |\eta_k(t)|^2 \leq c \dt,
\end{split}
\end{equation}
for some constant $c=c(T^*)>0$.

We therefore have,
\begin{equation}\label{4.41}
\begin{split}
&\lim_{k \to 0} \sup_{\{v_0,\te_0\}\in \mathcal{A}_k, \,nk
\in
[0,T^*]} |S_k^n \{v_0,\te_0\}-S(nk)\{v_0,\te_0\}|\\
&=\lim_{k \to 0} \sup_{\{v_0,\te_0\} \in \mathcal{A}_k, \,nk
\in
[0,T^*]} \sup_{\{v^n, \te^n\}\in S_k^n \{v_0,\te_0\}} | \{v^n,\te^n\}-\{v(nk),\te(nk)\}|\\
&=\lim_{k \to 0} \sup_{\{v_0,\te_0\} \in \mathcal{A}_k, \,nk
\in
[0,T^*]}\sup_{\{v^n, \te^n\}\in S_k^n \{v_0,\te_0\}} |\{\tilde{v}_k(nk), \tilde{\te}_k(nk)\}-\{v(nk),\te(nk)\}|\\
&=\lim_{k \to 0} \sup_{\{v_0,\te_0\} \in \mathcal{A}_k, \,nk
\in [0,T^*]} \sup_{\{v^n, \te^n\}\in S_k^n \{v_0,\te_0\}}|\{\xi_k(nk), \eta_k(nk)\}|=0,
\end{split}
\end{equation}
which concludes the proof of the lemma.
\end{proof}

We have, therefore, proved that conditions (H1) and (H2) of Theorem \ref{teo:approx} are both satisfied and thus, the long-term behavior of the semigroup $S(t)$ generated by the
continuous thermohydraulics equations \eqref{1.11}--\eqref{1.12}
is approximated 
by that of the  m-semigroups
generated by the discrete system \eqref{2.45a}--\eqref{1.26}. More precisely, we have the following:

\begin{Thm}\label{teo:apprrrr}
The family of attractors $\{\A_k\}_{k\in(0,\kappa_0]}$  converges, as $k\to 0$, to  $\A$, in the following sense:
$$
\lim_{k\to 0}\dist(\A_k,\A)=0,
$$
where $\dist$ denotes the Hausdorff semidistance in $H$, namely
$$
\dist(\A_k,\A)=\sup_{x_k\in\A_k}\inf_{x\in\A}|x_k-x|.
$$
\end{Thm}

\bigskip
\noindent{\bf Acknowledgements.}
This work was partially supported by the National Science Foundation under the grant NSF--DMS--0906440.


\bibliographystyle{siam}
\bibliography{References}

\end{document}